\documentclass[a4paper,11pt, oneside]{article}
\usepackage[ansinew]{inputenc}
\usepackage{graphicx}
\usepackage{xcolor}
\usepackage{amsfonts} 
\usepackage{amsthm} 
\usepackage{amsmath} 
\usepackage{amssymb}
\newcommand{\R}{{\mathbb R}}

\DeclareMathOperator{\diver}{div}

\usepackage{color}
\usepackage{hyperref}

\newtheorem{theorem}{Theorem}
\newtheorem{lemma}{Lemma}
\newtheorem{remark}{Remark}

\begin{document}

\title{Asymptotic justification of the Reynolds equation for a spherical bearing}

\author{Guy Bayada${}^{\ref{dir4}}$, Jos\'e M. Rodr\'{\i}guez${}^{\ref{dir1}, \ref{dir2}}$, Raquel Taboada-V\'azquez${}^{\ref{dir1}, \ref{dir3}}$}
\date{}

\maketitle

{\footnotesize 
\begin{enumerate}
\item INSA Lyon, Centrale Lyon, Universit\'e Lyon 1, Universit\'e Jean Monnet, CNRS, ICJ UMR5208, 69621 Villeurbanne, France. E-mail: guy.bayada@insa-lyon.fr. \label{dir4}
\item CITMAga, Santiago de Compostela, Spain. \label{dir1}
\item Department of Mathematics, Higher Technical University College of Architecture, Universidade da Coru\~na, A Coru\~na, Spain. E-mail: jose.rodriguez.seijo@udc.es. \label{dir2}
\item Department of Mathematics, School of Civil Engineering, Universidade da Coru\~na, A Coru\~na, Spain. E-mail: raquel.taboada@udc.es. \label{dir3}
\end{enumerate}
}

\begin{abstract}

To our knowledge, there is no rigorous mathematical justification of the Reynolds equation for a spherical bearing. In this article, we demonstrate that the solution of the Stokes problem in a domain between two closely spaced spheres converges, as the distance between the spheres approaches zero, to the solution of a ``Reynolds equation".

\end{abstract}

\section{Introduction}

In 1886, Reynolds introduced a model describing the behavior of a thin fluid layer in a lubrication problem (see \cite{Reynolds}). In deriving his model, Reynolds assumed that one of the two surfaces confining the fluid was flat. Although this assumption is locally acceptable, it may not be valid in many practical cases, leading several authors to propose alternative versions of the Reynolds equation, in which neither of the bounding surfaces is flat (see, for example, \cite{AskariAndersen2019}, \cite{BayadaChambat1990}, \cite{Chaoetal2018}, \cite{Elrod}, \cite{GoenkaBooker1980}, \cite{Jin2006}, \cite{Hakwoon2012}, \cite{Meyer2002}, \cite{Zhang2024}). In \cite{RodTab2022}-\cite{RodTab2022b}, two of the authors of this paper presented a new lubrication model applicable to any two boundary surfaces, using the method of asymptotic expansions to derive it formally. In \cite{BRT1}, the authors of this paper apply the results of \cite{RodTab2022}-\cite{RodTab2022b} to the case of cylindrical, spherical, and conical bearings, comparing the resulting models with others that can be found in the literature (for example, those already mentioned \cite{BayadaChambat1990}, \cite{GoenkaBooker1980},  \cite{Jin2006}, \cite{Hakwoon2012}, \cite{Meyer2002}).

Since most of the methods used in the literature to derive lubrication models are formal, several authors have attempted to justify these models more rigorously (for example, by proving that the solution of the Stokes problem converges, in a certain sense, to the solution of the lubrication model). Several justifications of the classical Reynolds equation can be found in \cite{BayadaChambat1986}, \cite{Cimatti1983} and \cite{Cimatti1987}, while \cite{Ciuperca2018} establishes the validity of the compressible Reynolds equation. Additional convergence results are mentioned in \cite{BayadaVazquez2007}. 

However, rigorous justifications of the Reynolds equation when neither of the confining surfaces is flat are scarce in the literature. A justification for a journal bearing is presented in \cite{BayadaChambat1990}, and \cite{Ghosh2021} provides a justification for a flow through a curved pipe. To our knowledge, no rigorous justification of the Reynolds equation between two arbitrary surfaces is currently available. 

In this article, we present a rigorous proof that, in the case of a viscous fluid confined between two closely spaced spheres, the solution of the Stokes problem converges, in a sense that we will specify, to the solution of the Reynolds equation \eqref{eq-det-p-17-a}--\eqref{eq-det-p-18}.

In section~\ref{section-step-1}, we define the domain where the Stokes problem is posed, namely the region between two closely spaced spheres. This domain depends on a small dimensionless parameter $\varepsilon$ related to the distance between the two spheres, which is assumed to tend to zero. We map the problem onto a reference domain independent of $\varepsilon$ by means of a change of variable. 

In section~\ref{section-step-2} we formulate the Stokes problem in the domain defined in section~\ref{section-step-1}. In section~\ref{section-step-3} we define the function spaces, and their norms, that will be used to study the convergence of the solution of the Stokes problem. 

Section~\ref{section-step-4} is dedicated to proving the existence of a divergence-free velocity field, which satisfies the boundary condition \eqref{eq-Stokes-3} and, also, the estimate \eqref{eq-estimate-1}. In sections~\ref{section-step-5} and \ref{section-step-6} we demonstrate the convergence of $\vec{u}^\varepsilon$ and $p^\varepsilon$, where $\vec{u}^\varepsilon$ and $p^\varepsilon$ denote, respectively, the velocity field and pressure of the Stokes problem posed in section~\ref{section-step-2}. Finally, in section~\ref{section-step-7}, we determine their respective limits.

\section{Working domain and a reference domain} \label{section-step-1}

We wish to point out that, in this article, the word ``sphere" should be understood as a simplification, and that the proposed procedure can be extended without major modifications to any two close surfaces defined in spherical coordinates by an expression of the type $\rho = \rho(\varphi, \theta)$.

In most of the lubricated devices, the outer ring (the spherical housing) must be able to assemble over the inner spherical ball, and standard spherical bearings do not reach $180^\circ$ coverage, typically covering around $120^\circ$ - $150^\circ$ of the inner spherical surface (see chapter 12 in \cite{Freneetal1997}, \cite{Jin2006}, \cite{Hakwoon2012}). Therefore, in the study that follows, we will only consider a sector of the inner sphere, which is assumed  to have radius $R$, and which can be parameterized in cartesian coordinates, using spherical coordinates $(\varphi, \theta, \rho)$, as follows:

\begin{equation}
\vec{X}(\varphi, \theta) = R (\sin \varphi \cos \theta, \sin \varphi \sin \theta, \cos \varphi), \quad (\varphi, \theta) \in D, \label{eq-inner-sphere} 
\end{equation}
where 
\begin{equation}
D= \left \{ (\varphi, \theta)\, /\, \varphi \in (\varphi_0,\varphi_1), \ 
\theta \in (0,2\pi) \right \}, \quad 0 < \varphi_0 < \varphi_1 < \pi. \label{eq-dominio-D}
\end{equation}

\begin{remark} \label{remark-bound-sin}

Assumption \eqref{eq-dominio-D} about the angle $\varphi$ implies that there exists $c_0 > 0$ such that $\sin \varphi \ge c_0$ 
for all $\varphi \in (\varphi_0, \varphi_1)$.
\end{remark}

\begin{remark} \label{remark-supply-holes}
As in the case of flat geometries, various boundary conditions can be considered, involving both pressure and velocities (see \cite{BRT1}), and they are mainly related to operational conditions, the geometry of the bearing and the location of the supply holes. 

From a physical point of view, the regions $\varphi = \varphi_0$ and $\varphi = \varphi_1$ define holes through which lubricant is supplied to the device (see \eqref{eq-g-4}--\eqref{eq-g-5}, Figure \ref{spherical-bearing}, and \cite{Freneetal1997} or \cite{Hakwoon2012}).

\end{remark}

\begin{figure}[htbp]
\vspace{-0.7cm}
\begin{center}
\includegraphics[width=12cm]{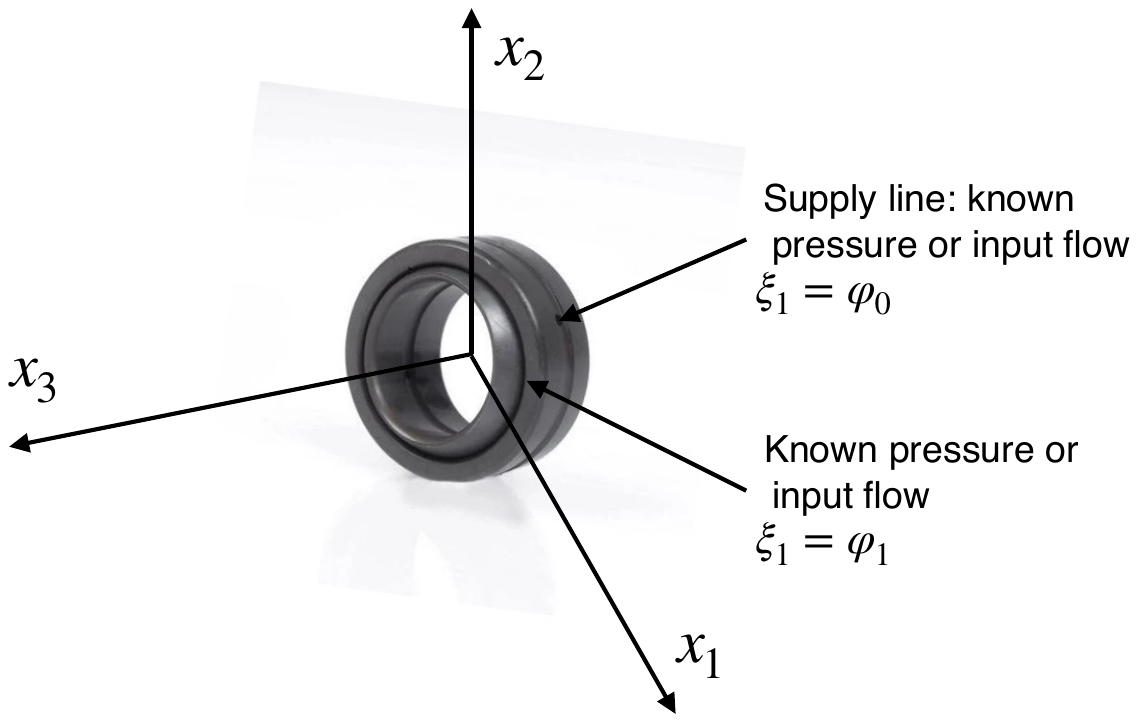}
\vspace{-1.5cm}
\caption{Example of spherical bearing}
\label{spherical-bearing}
\end{center}
\end{figure}

The gap between the inner sphere and the outer sphere is given by
\begin{align}
&h^\varepsilon(\varphi, \theta) = \varepsilon h(\varphi, \theta), \quad 0 < h_0 \le h(\varphi, \theta) \le h_1, \\
&h \textrm{ is smooth and } 2\pi \textrm{-periodic in the variable } \theta 
\end{align}

Let us define the unit normal to the sphere \eqref{eq-inner-sphere} as
\begin{equation}
\vec{N}(\varphi, \theta) = (\sin \varphi \cos \theta, \sin \varphi \sin \theta, \cos \varphi)
\end{equation}
then, the domain between the two spheres is defined by
\begin{align}
\Omega^\varepsilon &= \{ \vec{x} = (x_1, x_2, x_3) \in \R^3 
/ \nonumber \\
&\quad X_i(\varphi, \theta) \le x_i \le X_i(\varphi, \theta) + \xi_3 h^\varepsilon(\varphi, \theta) N_i(\varphi, \theta), 
\quad (i=1,2,3), \nonumber \\
&\quad (\varphi, \theta) \in D, \xi_3 \in (0,1) \} \label{eq-domain}
\end{align}

In the above equations $\varepsilon>0$ denotes a small dimensionless parameter that we will make tend to zero, indicating that the distance between the two spheres is very small.

\begin{figure}[htbp]
\vspace{-0.3cm}
\begin{center}
\includegraphics[width=6cm]{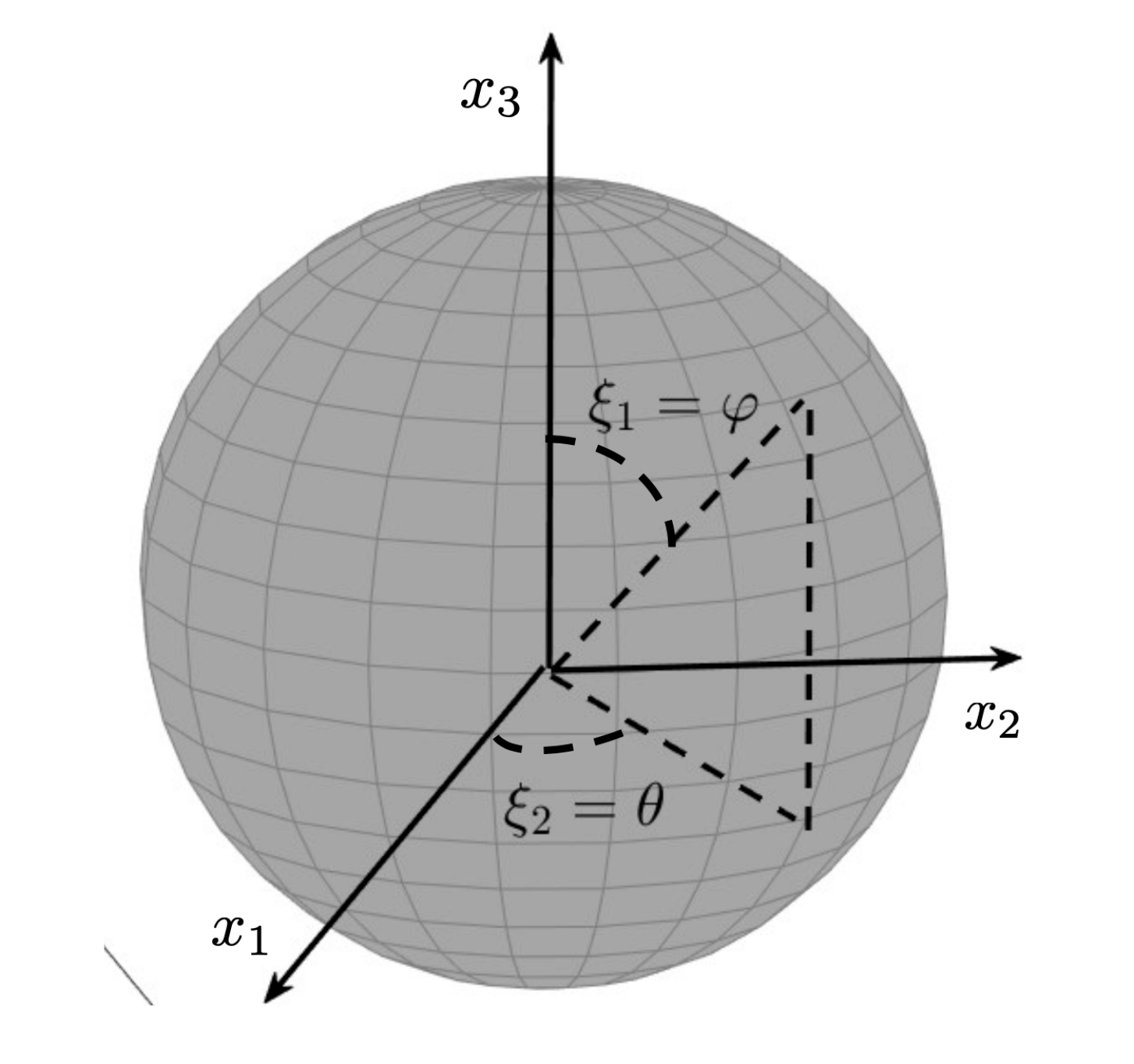}
\caption{Parameters $\xi_1, \xi_2$}
\label{parameters-sphere}
\end{center}
\end{figure}

\begin{remark}
Denoting $\vec{\xi} = (\xi_1, \xi_2, \xi_3)$, with $\xi_1 = \varphi$ and $\xi_2=\theta$, the domain \eqref{eq-domain} can also be defined by means of a change of variables 
\begin{equation}
\vec{x} = \Phi^\varepsilon(\vec{\xi}) 
= \vec{X}(\xi_1, \xi_2) + \varepsilon \xi_3 h(\xi_1,\xi_2) \vec{N}(\xi_1, \xi_2) \label{eq-change-variable}
\end{equation}
so we have that
\begin{equation}
\Omega^\varepsilon = \Phi^\varepsilon(\Omega), \quad \Omega = D \times (0,1) 
\end{equation}
For the steps that follow, it will be useful to decompose the previous change of variables into the following simpler ones:
\begin{align}
&\vec{\zeta} = (\zeta_1, \zeta_2, \zeta_3) = 
\Phi_1^\varepsilon(\vec{\xi}) = (\xi_1, \xi_2, \varepsilon \xi_3), \\
&(\varphi, \theta, \rho) = \Phi_2^\varepsilon(\vec{\zeta}) = 
(\zeta_1, \zeta_2, R + \zeta_3 h(\zeta_1, \zeta_2)), \\
&\vec{x} = \Phi_3^\varepsilon(\varphi, \theta, \rho) = 
(\rho \sin \varphi \cos \theta, \rho \sin \varphi \sin \theta, \rho \cos \varphi), 
\end{align}
Thus, the change of variables \eqref{eq-change-variable} can be written 
\begin{equation}
\Phi^\varepsilon = \Phi_3^\varepsilon \circ \Phi_2^\varepsilon \circ \Phi_1^\varepsilon
\end{equation}
and the domain \eqref{eq-domain} is expressed, for each of the different variables,
\begin{align}
&\Omega_i^\varepsilon = \Phi_i^\varepsilon(\Omega_{i-1}^\varepsilon), \qquad (i=1,2,3), \\
&\Omega_0^\varepsilon = \Omega, 
\end{align}
so $\Omega_0^\varepsilon = \Omega$ does not depend on $\varepsilon$, 
and $\Omega_3^\varepsilon$ is exactly $\Omega^\varepsilon$. 
More precisely, the domains $\Omega_i^\varepsilon \ (i=0,1,2,3)$ can be defined as: 
\begin{align}
&\Omega_0^\varepsilon = \Omega = \{ (\xi_1, \xi_2, \xi_3) / \xi_1 \in (\varphi_0,\varphi_1), \xi_2 \in (0,2\pi), 
\xi_3 \in (0,1) \} \\
&\Omega_1^\varepsilon = \{ (\zeta_1, \zeta_2, \zeta_3) / \zeta_1 \in (\varphi_0,\varphi_1), \zeta_2 \in (0,2\pi), 
\zeta_3 \in (0,\varepsilon) \} \\
&\Omega_2^\varepsilon = \{ (\varphi, \theta, \rho) / \varphi \in (\varphi_0,\varphi_1), \theta \in (0,2\pi), 
\rho \in (R, R + \varepsilon h(\varphi, \theta)) \} \\
&\Omega_3^\varepsilon = \Omega^\varepsilon \textrm{ given by \eqref{eq-domain}} 
\end{align}
\end{remark}

\section{Stokes problem} \label{section-step-2}

Let us consider de Stokes problem posed in $\Omega^\varepsilon$:

\begin{align}
- \mu \Delta \vec{u}^\varepsilon + \nabla p^\varepsilon &= \vec{0} \quad \textrm{in } \Omega^\varepsilon 
\label{eq-Stokes-1} \\
\diver \vec{u}^\varepsilon &=  0 \quad \textrm{in } \Omega^\varepsilon \label{eq-Stokes-2} \\
\vec{u}^\varepsilon &= \vec{g}^\varepsilon \quad \textrm{on } \partial \Omega^\varepsilon \label{eq-Stokes-3}
\end{align}
where $\mu > 0$ and $\vec{g}^\varepsilon \in \left ( H^{1/2}(\partial \Omega^\varepsilon) \right )^3$. The proof of the following Theorem, which will be used in the sequel, can be found in \cite{NSTemam}.

\begin{theorem} \label{theorem-existence-stokes-1}
Problem \eqref{eq-Stokes-1}--\eqref{eq-Stokes-3} has an unique solution $\vec{u}^\varepsilon \in \left ( H^1(\Omega^\varepsilon) \right )^3$, $p^\varepsilon \in L_0^2(\Omega^\varepsilon) = \left \{ \phi \in 
L^2(\Omega^\varepsilon) / \int_{\Omega^\varepsilon} \phi\, d\vec{x} = 0 \right \}$, if the following compatibility condition is satisfied: 
\begin{equation}
\int_{\partial \Omega^\varepsilon} \vec{g}^\varepsilon \cdot \vec{n}^\varepsilon dA = 0 \label{eq-cond-comp}
\end{equation}
where $\vec{n}^\varepsilon$ is the outer unit normal to $\partial \Omega^\varepsilon$ and $dA$ is the differential element of area. Furthermore, there exists a constant $C^\varepsilon$ (depending on $\mu$ and $\Omega^\varepsilon$) such that 
\begin{equation}
{\| \vec{u}^\varepsilon \|}_{\left ( H^1(\Omega^\varepsilon) \right )^3} + 
{\| p^\varepsilon \|}_{L_0^2(\Omega^\varepsilon)} \le C^\varepsilon {\| \vec{g}^\varepsilon \|}_{\left ( H^{1/2}(\partial \Omega^\varepsilon) \right )^3}
\end{equation}

\end{theorem}

Moreover, the following weak formulation is valid: 

\begin{theorem}
If the compatibility condition \eqref{eq-cond-comp} is satisfied, then 
there exists an unique solution $\vec{u}^\varepsilon \in V = \left \{ \vec{v} \in \left ( H^1(\Omega^\varepsilon) \right )^3 
/ \diver \vec{v} = 0 \right \}$, $p^\varepsilon \in L_0^2(\Omega^\varepsilon)$ of the problem 
\begin{equation}
\mu \int_{\Omega^\varepsilon} \nabla \vec{u}^\varepsilon \cdot \nabla \vec{\psi} \, d\vec{x} = 
\int_{\Omega^\varepsilon} p^\varepsilon \diver \vec{\psi} \, d\vec{x}, \qquad 
\forall \vec{\psi} \in \left ( H^1_0(\Omega^\varepsilon) \right )^3 \label{eq-Stokes-var}
\end{equation}
such that $\vec{u}^\varepsilon = \vec{g}^\varepsilon$ on $\partial \Omega^\varepsilon$ in the sense of traces.
\end{theorem}

\section{Relations between the norms in certain function spaces} \label{section-step-3}

Given a scalar function $\phi$ defined on $\Omega^\varepsilon$, we denote by $\phi_{\Omega_i^\varepsilon}$ the corresponding function on $\Omega_i^\varepsilon$ obtained via the change of variables. Thus,
\begin{equation}
\phi_{\Omega_i^\varepsilon}(\vec{y}^{\, i}) = \phi(\vec{x}), \quad (i=0,1,2,3),
\end{equation}
where 
\begin{equation}
\vec{y}^{\, 0} = \vec{\xi}, \quad \vec{y}^{\, 1} = \vec{\zeta}, \quad \vec{y}^{\, 2} = (\varphi, \theta, \rho), 
\quad \vec{y}^{\, 3} = \vec{x}.
\end{equation}

If there is no possibility of confusion, we will omit the subscript and simply write
\begin{equation}
\phi(\vec{y}^{\, i}) = \phi(\vec{x}), \quad (i=0,1,2),
\end{equation}
or only $\phi$.

For vector functions, in addition to the change of variables, the change of basis must also be taken into account. For example, 
\begin{align}
\vec{u} &= u_1(\vec{x})\vec{e}_1 + u_2(\vec{x})\vec{e}_2 + u_3(\vec{x})\vec{e}_3 
\quad \textrm{in } \Omega^\varepsilon \\
&= u_{\varphi, \Omega_2^\varepsilon}(\vec{y}^{\, 2})\vec{e}_\varphi 
+ u_{\theta, \Omega_2^\varepsilon}(\vec{y}^{\, 2})\vec{e}_\theta 
+ u_{\rho, \Omega_2^\varepsilon}(\vec{y}^{\, 2})\vec{e}_\rho 
\quad \textrm{in } \Omega_2^\varepsilon 
\end{align}
where $\{ \vec{e}_1, \vec{e}_2, \vec{e}_3 \}$ is an orthonormal basis of $\R^3$ and 
\begin{align}
\vec{e}_\varphi&=(\cos\xi_1\cos\xi_2,\cos \xi_1\sin\xi_2,-\sin\xi_1) \label{base_e_phi} \\
\vec{e}_\theta&=(-\sin\xi_2,\cos\xi_2,0) \label{base_e_theta} \\
\vec{e}_\rho &= (\sin\xi_1\cos\xi_2,\sin \xi_1\sin\xi_2,\cos\xi_1) \label{base_e_rho} 
\end{align}
where we recall that $\xi_1 = \zeta_1 = \varphi$ and $\xi_2 = \zeta_2 = \theta$.

Let us consider the outer surface of the domain in the different systems of coordinates:
\begin{align}
&S_0^+ =  \{ (\xi_1, \xi_2, \xi_3) \in \partial \Omega = \partial \Omega_0^\varepsilon 
/ \xi_3 = 1 \} \\
&S_1^+ = \{ (\zeta_1, \zeta_2, \zeta_3) \in \partial \Omega_1^\varepsilon / \zeta_3  = \varepsilon \} \\
&S_2^+ =  \{ (\varphi, \theta, \rho) \in \partial \Omega_2^\varepsilon / \rho = R + \varepsilon h(\varphi, \theta) \} \\
&S_3^+ =  \{ (x_1, x_2, x_3) \in \partial \Omega^\varepsilon = \partial \Omega_3^\varepsilon 
/ x_1^2 + x_2^2 + x_3^2 = (R + \varepsilon h(\varphi, \theta))^2 \}
\end{align}

We now introduce the following function spaces $(i=0,1,2,3)$:

\begin{equation}
V_i = \left \{ \psi \in H^1(\Omega_i^\varepsilon) / \psi = 0 \textrm{ on } S_i^+ \right \}  
\end{equation}
and the following norms 
\begin{align}
&\left | \psi \right |_i = \left ( \int_{\Omega_i^\varepsilon} \psi^2\, d\vec{y}^{\, i} \right )^{1/2} \\
&\left \| \psi \right \|_i = \left ( \sum_{j=1}^3 \int_{\Omega_i^\varepsilon} \left ( \frac{\partial \psi}{\partial y_j^i} \right )^2\, d\vec{y}^{\, i} \right )^{1/2} 
\end{align}
\begin{remark}
In what follows, we will denote by $C$ or $C_i$ various constants independent of $\varepsilon$, which may depend on other data (such as $\Omega$, $\mu$, $h$, \dots), although not explicitly mentioned.
\end{remark}
\begin{remark}
It is possible to prove the following Poincar\'e inequality: There exists a constant $C$ independent of $\varepsilon$ such that 
\begin{equation}
\left | \psi \right |_i \le C \left \| \psi \right \|_i \qquad (i=0,1,2,3) \label{eq-poincare}
\end{equation}
so the norm $\left \| \cdot \right \|_i$ is equivalent to the $H^1(\Omega_i^\varepsilon)$ norm. In fact, for $i=1,2,3$, 
since the thickness of $\Omega_i^\varepsilon$ is of the order of $\varepsilon$, it can be proved that there exists a constant $C$ independent of $\varepsilon$ such that
\begin{equation}
\left | \psi \right |_i \le C \varepsilon \left \| \psi \right \|_i \qquad (i=1,2,3)
\end{equation}
\end{remark}

Using Poincar\'e inequality \eqref{eq-poincare} and elementary calculus, we obtain the following result:
\begin{lemma}
There exist positive constants $C_k >0 \ (k \in \{ 1,2, \dots, 8 \})$, independent of $\varepsilon$, such that
\begin{align}
&C_1 \left | \psi \right |_j \le \left | \psi \right |_i \le C_2 \left | \psi \right |_j \qquad (i,j=1,2,3) \label{eq-ineq-1} \\ 
&C_3 \left \| \psi \right \|_j \le \left \| \psi \right \|_i \le C_4 \left \| \psi \right \|_j \qquad (i,j=1,2,3) \label{eq-ineq-2} \\
&\frac{C_5}{\sqrt{\varepsilon}} \left | \psi \right |_j \le \left | \psi \right |_0 \le 
\frac{C_6}{\sqrt{\varepsilon}} \left | \psi \right |_j \qquad (j=1,2,3) \label{eq-ineq-3} \\
&C_7\sqrt{\varepsilon} \left \| \psi \right \|_j \le \left \| \psi \right \|_0 \le 
\frac{C_8}{\sqrt{\varepsilon}} \left \| \psi \right \|_j \qquad (j=1,2,3) \label{eq-ineq-4}
\end{align}
for any $\psi \in V_i \ (i=1,2,3)$.
\end{lemma}

Obviously, we have the same result for vector functions $\vec{\psi} \in (V_i)^3 \ (i=0,1,2,3)$, where we define:
\begin{align}
&\left | \vec{\psi} \right |_i = \left ( \sum_{j=1}^3 \left | \psi_j \right |_i^2 \right )^{1/2} \\
&\left \| \vec{\psi} \right \|_i = \left ( \sum_{j=1}^3 \left \| \psi_j \right \|_i^2 \right )^{1/2}
\end{align}
that is,
\begin{lemma}
There exist positive constants $C_k >0 \ (k \in \{ 1,2, \dots, 8 \})$, independent of $\varepsilon$, such that
\begin{align}
&C_1 \left | \vec{\psi} \right |_j \le \left | \vec{\psi} \right |_i \le C_2 \left | \vec{\psi} \right |_j \qquad (i,j=1,2,3) \label{eq-ineq-vec-1} \\ 
&C_3 \left \| \vec{\psi} \right \|_j \le \left \| \vec{\psi} \right \|_i \le C_4 \left \| \vec{\psi} \right \|_j \qquad (i,j=1,2,3) \label{eq-ineq-vec-2} \\
&\frac{C_5}{\sqrt{\varepsilon}} \left | \vec{\psi} \right |_j \le \left | \vec{\psi} \right |_0 \le 
\frac{C_6}{\sqrt{\varepsilon}} \left | \vec{\psi} \right |_j \qquad (j=1,2,3) \label{eq-ineq-vec-3} \\
&C_7\sqrt{\varepsilon} \left \| \vec{\psi} \right \|_j \le \left \| \vec{\psi} \right \|_0 \le 
\frac{C_8}{\sqrt{\varepsilon}} \left \| \vec{\psi} \right \|_j \qquad (j=1,2,3) \label{eq-ineq-vec-4}
\end{align}
for any $\vec{\psi} \in (V_i)^3 \ (i=1,2,3)$.
\end{lemma}

\section{Boundary conditions and divergence free field} \label{section-step-4}

The objective of this section is to prove the existence of a divergence free lifting of the boundary condition \eqref{eq-Stokes-3}, whose norm admits an explicit bound in terms of $\varepsilon$. Specifically, we establish the existence of at least one solution $\vec{v}^\varepsilon \in \left ( H^1(\Omega^\varepsilon) \right )^3$ of 
\begin{align}
\diver \vec{v}^\varepsilon &=  0 \quad \textrm{in } \Omega^\varepsilon \label{eq-divergence-1} \\
\vec{v}^\varepsilon &= \vec{g}^\varepsilon \quad \textrm{on } \partial \Omega^\varepsilon \label{eq-divergence-2}
\end{align}
such that 
\begin{equation}
{\| \vec{v}^\varepsilon \|}_{\left ( H^1(\Omega^\varepsilon) \right )^3} \le \frac{C}{\sqrt{\varepsilon}}  \label{eq-estimate-1}
\end{equation}
where $C>0$ is independent of $\varepsilon$.

From Theorem \ref{theorem-existence-stokes-1}, we deduce the existence of a solution of problem 
\eqref{eq-divergence-1}--\eqref{eq-divergence-2} satisfying 
\begin{equation}
{\| \vec{v}^\varepsilon \|}_{\left ( H^1(\Omega^\varepsilon) \right )^3} \le C^\varepsilon {\| \vec{g}^\varepsilon \|}_{\left ( H^{1/2}(\partial \Omega^\varepsilon) \right )^3}
\end{equation}
with $C^\varepsilon$ depending on $\mu$ and $\Omega^\varepsilon$, that is, with $C^\varepsilon$ possibly dependent on $\varepsilon$. 

To prove \eqref{eq-estimate-1} we rewrite problem \eqref{eq-divergence-1}--\eqref{eq-divergence-2} in $\Omega$ and make additional assumptions on the boundary condition $\vec{g}^\varepsilon$. 

So far, we have only assumed that $\vec{g}^\varepsilon$ verifies \eqref{eq-cond-comp}. We now impose further hypotheses on  $\vec{g}^\varepsilon$ on $\partial \Omega^\varepsilon$.

Let us suppose that the inner sphere rotates about the $Oz$ axis with angular velocity $\omega$, while the outer sphere is stationary. In that case (for convenience, we specify the values of $\vec{g}^\varepsilon$ on the boundary of $\Omega$ in spherical coordinates):

\begin{align}
\vec{g}_{\partial \Omega}^\varepsilon &= R\omega \sin \xi_1 \vec{e}_\theta \textrm{ on } \xi_3 = 0, \label{eq-g-1}\\
\vec{g}_{\partial \Omega}^\varepsilon &= \vec{0}  \textrm{ on } \xi_3 = 1. \label{eq-g-2}
\end{align}

Since the sphere is closed in the $\theta$-direction, we have that 

\begin{equation}
\vec{g}_{\partial \Omega}^\varepsilon \textrm{ is $2\pi$ periodic in $\xi_2$-direction}. \label{eq-g-3}
\end{equation}

Finally, we assume that, on the surfaces $\varphi = \varphi_0$ and $\varphi = \varphi_1$, the fluid also rotates about the axis $Oz$, in a way that these boundary conditions are compatible with \eqref{eq-g-1}--\eqref{eq-g-2}. We allow for inflow and outflow through these two boundaries (supply conditions):
\begin{align}
\vec{g}_{\partial \Omega}^\varepsilon &= k_0^\varepsilon(\xi_2, \xi_3)\vec{e}_\varphi + R\omega \sin \xi_1 g_0(\xi_3)\vec{e}_\theta \textrm{ on } \xi_1 = \varphi_0,  
\label{eq-g-4} \\
\vec{g}_{\partial \Omega}^\varepsilon &= k_1^\varepsilon(\xi_2, \xi_3)\vec{e}_\varphi + R\omega \sin \xi_1 g_1(\xi_3)\vec{e}_\theta \textrm{ on } \xi_1 = \varphi_1, 
\label{eq-g-5}
\end{align}
where $g_0$ and $g_1$ are smooth functions of $\xi_3$, satisfying
\begin{equation}
g_0(0) = 1, \ g_0(1) = 0, \ g_1(0) = 1, \ g_1(1) = 0, \label{eq-g-5-b}
\end{equation}
and $k_0^\varepsilon$, $k_1^\varepsilon$ are smooth functions of $(\xi_2, \xi_3)$, 
$2\pi$ periodic in $\xi_2$, such that  
\begin{equation}
k_0^\varepsilon(\xi_2, 0) = k_0^\varepsilon(\xi_2, 1) = k_1^\varepsilon(\xi_2, 0) = k_1^\varepsilon(\xi_2, 1) = 0,
\label{eq-g-5-c}
\end{equation}
and 
\begin{equation}
| k_0^\varepsilon |_{L^\infty(\Omega^\varepsilon)} \le C,  \ | k_1^\varepsilon |_{L^\infty(\Omega^\varepsilon)} \le C. \label{eq-g-5-d}
\end{equation}

It is easy to check that, under hypotheses \eqref{eq-g-1}--\eqref{eq-g-5-d}, $\vec{g}^\varepsilon$ fulfills \eqref{eq-cond-comp}, if and only if 
\begin{equation}
\int_{\varphi=\varphi_1} k_1^\varepsilon \ dA - \int_{\varphi=\varphi_0} k_0^\varepsilon \ dA = 0 \label{eq-g-6} 
\end{equation}
that under a change of variables, can be written as 
\begin{align}
&\varepsilon \iint_{D_\varphi} \left [ h(\varphi_1,\xi_2) \left ( R + \varepsilon \xi_3 h(\varphi_1,\xi_2) \right ) 
k_1^\varepsilon(\xi_2, \xi_3) \sin \varphi_1 \right . \nonumber \\
&\quad{} \left . {} - 
h(\varphi_0,\xi_2) \left ( R + \varepsilon \xi_3 h(\varphi_0,\xi_2) \right ) 
k_0^\varepsilon(\xi_2, \xi_3) \sin \varphi_0 \right ] \, d\xi_2 d\xi_3 = 0 \label{eq-g-7} 
\end{align}
where 
\begin{equation}
D_\varphi = \left \{ (\xi_2, \xi_3) / \xi_2 \in (0, 2\pi), \xi_3 \in (0,1) \right \}. \label{eq-def-D-phi}
\end{equation}

Consequently, if $\vec{g}^\varepsilon$ verifies \eqref{eq-g-1}--\eqref{eq-g-5-d}, \eqref{eq-cond-comp}, \eqref{eq-g-6} and \eqref{eq-g-7} are equivalent.

\begin{lemma} \label{lemma-K}
Let us assume \eqref{eq-g-1}--\eqref{eq-g-5-d} and \eqref{eq-g-6}. Then 
there exists $\vec{K} \in \left ( H^1(\Omega) \right )^3$ such that 
\begin{align}
\diver \vec{K} &=  0 \quad \textrm{in } \Omega \label{eq-divergence-K-1} \\
\vec{K} &= \vec{G}^\varepsilon \quad \textrm{on } \partial \Omega \label{eq-divergence-K-2}
\end{align}
and 
\begin{align}
&{\| \vec{K} \|}_0 \le C_1 {\| \vec{G}^\varepsilon \|}_{\left ( H^{1/2}(\partial \Omega) \right )^3} \label{eq-estimate-K-1} \\
&{\| \vec{K} \|}_0 \le C_2  \label{eq-estimate-K-2} 
\end{align}
where $C_1>0$ and $C_2>0$ are independent of $\varepsilon$, and where 
\begin{align}
\vec{G}^\varepsilon &= R^2 h \omega \sin\xi_1 \vec{e}_\theta \textrm{ on } \xi_3 = 0 \label{eq-G-1}\\
\vec{G}^\varepsilon &= \vec{0}  \textrm{ on } \xi_3 = 1 \label{eq-G-2} \\
\vec{G}^\varepsilon &\textrm{ is $2\pi$ periodic in $\xi_2$} \label{eq-G-3} \\
\vec{G}^\varepsilon &= (R+\varepsilon \xi_3 h) k_0^\varepsilon \sin \varphi_0 
\left [ h \vec{e}_\varphi - \xi_3 \frac{\partial h}{\partial \xi_1} \vec{e}_\rho \right ] \nonumber \\
&{} + (R+\varepsilon \xi_3 h) R\omega g_0(\xi_3) \sin \varphi_0 \left [ h \vec{e}_\theta 
- \xi_3 \frac{\partial h}{\partial \xi_2} \vec{e}_\rho \right ] 
\textrm{ on } \xi_1 = \varphi_0 \label{eq-G-4} \\
\vec{G}^\varepsilon &= (R+\varepsilon \xi_3 h) k_1^\varepsilon \sin \varphi_1 
\left [ h \vec{e}_\varphi - \xi_3 \frac{\partial h}{\partial \xi_1} \vec{e}_\rho \right ] \nonumber \\
&{} + (R+\varepsilon \xi_3 h) R\omega g_1(\xi_3) \sin \varphi_1 \left [ h \vec{e}_\theta 
- \xi_3 \frac{\partial h}{\partial \xi_2} \vec{e}_\rho \right ] 
\textrm{ on } \xi_1 = \varphi_1 \label{eq-G-5} 
\end{align}
\end{lemma}
\begin{proof}
In a first step, we write \eqref{eq-divergence-1} in spherical coordinates in $\Omega_2^\varepsilon$ (namely, in $(\varphi, \theta, \rho)$-coordinates), 
\begin{equation}
\frac{1}{\rho^2}\frac{\partial}{\partial \rho} \left ( \rho^2 v_\rho^\varepsilon \right ) + 
\frac{1}{\rho \sin \varphi}\frac{\partial}{\partial \varphi} \left ( v_\varphi^\varepsilon \sin \varphi \right ) + 
\frac{1}{\rho \sin \varphi}\frac{\partial v_\theta^\varepsilon}{\partial \theta} = 0 \label{eq-divergence-1-sph-b}
\end{equation}
and then we rewrite \eqref{eq-divergence-1-sph-b} in $(\xi_1, \xi_2, \xi_3)$-coordinates, that is, in $\Omega$, 
\begin{align}
&\frac{1}{\varepsilon h (R+\varepsilon \xi_3 h)^2}\frac{\partial}{\partial \xi_3} \left [ (R+\varepsilon \xi_3 h)^2 v_{\rho, \Omega}^\varepsilon \right ] \nonumber \\
&\quad {} + \frac{1}{(R+\varepsilon \xi_3 h) \sin \xi_1}\left [ \frac{\partial}{\partial \xi_1} \left ( v_{\varphi, \Omega}^\varepsilon \sin \xi_1 \right ) - \frac{\xi_3}{h}\frac{\partial h}{\partial \xi_1} \frac{\partial}{\partial \xi_3} \left ( v_{\varphi, \Omega}^\varepsilon \sin \xi_1 \right ) \right ] \nonumber \\
&\quad {} + 
\frac{1}{(R+\varepsilon \xi_3 h) \sin \xi_1}\left [ \frac{\partial v_{\theta, \Omega}^\varepsilon}{\partial \xi_2} 
- \frac{\xi_3}{h}\frac{\partial h}{\partial \xi_2} \frac{\partial v_{\theta, \Omega}^\varepsilon}{\partial \xi_3} \right ]
= 0 \label{eq-continuity-omega}
\end{align}

This last equation can be rewritten as
\begin{equation}
\frac{\partial J_1^\varepsilon}{\partial \xi_1} + \frac{\partial J_2^\varepsilon}{\partial \xi_2} 
+ \frac{\partial J_3^\varepsilon}{\partial \xi_3} = 0 
\qquad \textrm{in}\ \Omega \label{eq-divergence-omega}
\end{equation}
where
\begin{align}
J_1^\varepsilon &= h (R+\varepsilon \xi_3 h) v_{\varphi, \Omega}^\varepsilon \sin \xi_1 \label{eq-J-1} \\
J_2^\varepsilon &= h (R+\varepsilon \xi_3 h) v_{\theta, \Omega}^\varepsilon \label{eq-J-2} \\
J_3^\varepsilon &= \frac{1}{\varepsilon}(R+\varepsilon \xi_3 h)^2 v_{\rho, \Omega}^\varepsilon \sin \xi_1 
- \xi_3 (R+\varepsilon \xi_3 h) \frac{\partial h}{\partial \xi_1} v_{\varphi, \Omega}^\varepsilon \sin \xi_1 \nonumber \\
&\quad {} - \xi_3 (R+\varepsilon \xi_3 h) \frac{\partial h}{\partial \xi_2} v_{\theta, \Omega}^\varepsilon \label{eq-J-3}
\end{align}

Let us now observe that $\vec{G}^\varepsilon = \vec{J}^\varepsilon_{| \partial \Omega}$. 
From \eqref{eq-G-1}--\eqref{eq-G-5} and \eqref{eq-g-7}, we obtain that 
\begin{align}
&  \int_{\partial \Omega} \vec{J}^\varepsilon \cdot  \vec{n}\, dA =  \int_{\partial \Omega} \vec{G}^\varepsilon \cdot  \vec{n}\, dA \nonumber \\ 
& = \varepsilon \iint_{D_\varphi} \left [ h(\varphi_1,\xi_2) \left ( R + \varepsilon \xi_3 h(\varphi_1,\xi_2) \right ) 
k_1^\varepsilon(\xi_2, \xi_3) \sin \varphi_1 \right . \nonumber \\
&\quad{} \left . {} - 
h(\varphi_0,\xi_2) \left ( R + \varepsilon \xi_3 h(\varphi_0,\xi_2) \right ) 
k_0^\varepsilon(\xi_2, \xi_3) \sin \varphi_0 \right ] \, d\xi_2 d\xi_3 = 0 \label{eq-comp-G}
\end{align}

By \eqref{eq-divergence-omega} and \eqref{eq-J-1}--\eqref{eq-J-3}, $\vec{J}^\varepsilon$ satisfies
\eqref{eq-divergence-K-1}--\eqref{eq-divergence-K-2}. Moreover, 
as a consequence of \eqref{eq-comp-G}, Theorem \ref{theorem-existence-stokes-1} and the fact that $\Omega$ does not depend on $\varepsilon$, we deduce \eqref{eq-estimate-K-1}. Finally, from \eqref{eq-G-1}--\eqref{eq-G-5} we have that $\vec{G}^\varepsilon$ is bounded, which yields \eqref{eq-estimate-K-2}.

\end{proof}

\begin{theorem}
Let us assume \eqref{eq-g-1}--\eqref{eq-g-6}. Then there exists 
$\vec{r}^\varepsilon \in \left ( H^1(\Omega^\varepsilon) \right )^3$ satisfying 
\begin{align}
\diver \vec{r}^\varepsilon &=  0 \quad \textrm{in } \Omega^\varepsilon \label{eq-divergence-r-1} \\
\vec{r}^\varepsilon &= \vec{g}^\varepsilon \quad \textrm{on } \partial \Omega^\varepsilon \label{eq-divergence-r-2}
\end{align}
and 
\begin{equation}
\left \| \vec{r}^\varepsilon \right \|_3 \le \frac{C}{\sqrt{\varepsilon}} \label{eq-estimate-r}
\end{equation}
\end{theorem}

\begin{proof} Let $\vec{K} \in \left ( H^1(\Omega) \right )^3$ be given by Lemma \ref{lemma-K}, and 
let us consider the function defined on $\Omega$:
\begin{align}
\vec{R}^\varepsilon &= R_1^\varepsilon \vec{e}_\varphi + R_2^\varepsilon \vec{e}_\theta 
+ R_3^\varepsilon \vec{e}_\rho \label{eq-def-R-1} \\
R_1^\varepsilon &= \frac{K_1}{h (R+\varepsilon \xi_3 h) \sin \xi_1} \label{eq-def-R-2} \\
R_2^\varepsilon &= \frac{K_2}{h (R+\varepsilon \xi_3 h)} \label{eq-def-R-3} \\
R_3^\varepsilon &=\frac{\varepsilon}{(R+\varepsilon \xi_3 h)^2 \sin \xi_1} \left ( K_3 + 
\frac{\xi_3}{h}\frac{\partial h}{\partial \xi_1} K_1 + \frac{\xi_3}{h}\frac{\partial h}{\partial \xi_2} K_2 \right ) 
\label{eq-def-R-4}
\end{align}

We now introduce $\vec{R}^\varepsilon$ in $\Omega^\varepsilon$, that is, 
\begin{align}
\vec{r}^\varepsilon(\vec{x}) &= r_1^\varepsilon(\vec{x})\vec{e}_1 + r_2^\varepsilon(\vec{x})\vec{e}_2 
+ r_3^\varepsilon(\vec{x})\vec{e}_3 \quad \textrm{in } \Omega^\varepsilon \label{eq-def-r-1} \\
&=  R_1^\varepsilon(\vec{\xi})\vec{e}_\varphi + R_2^\varepsilon(\vec{\xi})\vec{e}_\theta 
+ R_3^\varepsilon(\vec{\xi})\vec{e}_\rho = \vec{R}^\varepsilon(\vec{\xi}) \quad \textrm{in } \Omega 
\label{eq-def-r-2}
\end{align}

To prove \eqref{eq-divergence-r-1}--\eqref{eq-divergence-r-2}, we follow the steps of Lemma \ref{lemma-K} in reverse order. 
Estimate \eqref{eq-estimate-r} follows from \eqref{eq-estimate-K-2}, the definitions 
\eqref{eq-def-R-1}--\eqref{eq-def-r-2}, and \eqref{eq-ineq-vec-4}.
\end{proof}

\section{Convergence of $\vec{u}^\varepsilon$} \label{section-step-5}

\begin{theorem}
Let $\vec{u}^\varepsilon$ be the solution of \eqref{eq-Stokes-1}--\eqref{eq-Stokes-3} with 
$\vec{g}^\varepsilon$ given by \eqref{eq-g-1}--\eqref{eq-g-5}, and let $\vec{r}^\varepsilon$ be defined by 
\eqref{eq-def-r-1}--\eqref{eq-def-r-2}. Then 
\begin{equation}
\left \| \vec{u}^\varepsilon \right \|_3 \le 2 \left \| \vec{r}^\varepsilon \right \|_3 \label{eq-estimate-u-r}
\end{equation}
\end{theorem}
\begin{proof}
Let  us define $\vec{w}^\varepsilon = \vec{u}^\varepsilon - \vec{r}^\varepsilon$. From 
\eqref{eq-Stokes-1}--\eqref{eq-Stokes-3} and \eqref{eq-divergence-r-1}--\eqref{eq-divergence-r-2} 
we deduce that 
\begin{align}
- \mu \Delta \vec{w}^\varepsilon + \nabla p^\varepsilon &= \mu \Delta \vec{r}^\varepsilon \quad \textrm{in } \Omega^\varepsilon  \label{eq-Stokes-w-1} \\
\diver \vec{w}^\varepsilon &=  0 \quad \textrm{in } \Omega^\varepsilon \label{eq-Stokes-w-2} \\
\vec{w}^\varepsilon &= \vec{0} \quad \textrm{on } \partial \Omega^\varepsilon \label{eq-Stokes-w-3}
\end{align}
Multiplying \eqref{eq-Stokes-w-1} by $\vec{w}^\varepsilon$, integrating over $\Omega^\varepsilon$, and applying the divergence theorem, taking into account that $\vec{w}^\varepsilon$ satisfies \eqref{eq-Stokes-w-2}--\eqref{eq-Stokes-w-3}, we obtain 
\begin{equation}
\mu \int_{\Omega^\varepsilon} \nabla w^\varepsilon \cdot \nabla w^\varepsilon\, d\vec{x} = 
- \mu \int_{\Omega^\varepsilon} \nabla r^\varepsilon \cdot \nabla w^\varepsilon\, d\vec{x} \label{eq-estimate-w-r}
\end{equation}

From \eqref{eq-estimate-w-r} it follows that
\begin{equation}
\left \| \vec{w}^\varepsilon \right \|_3^2 \le \left \| \vec{r}^\varepsilon \right \|_3 \left \| \vec{w}^\varepsilon \right \|_3
\end{equation}
and consequently 
\begin{equation}
\left \| \vec{w}^\varepsilon \right \|_3 \le \left \| \vec{r}^\varepsilon \right \|_3 
\Rightarrow \left \| \vec{u}^\varepsilon - \vec{r}^\varepsilon \right \|_3 \le \left \| \vec{r}^\varepsilon \right \|_3 
\Rightarrow \left \| \vec{u}^\varepsilon \right \|_3 \le 2 \left \| \vec{r}^\varepsilon \right \|_3 \label{eq-estimate-u-r-bis}
\end{equation}


\end{proof}

Let us introduce the Hilbert space defined by
\begin{equation}
V^* = \left \{ v \in L^2(\Omega) / \frac{\partial v}{\partial \xi_3} \in L^2(\Omega), v=0 \textrm{ on } \xi_3 = 1 \right \}
\end{equation}
with the norm 
\begin{equation}
\left \| v \right \|_{V^*} = \left | \frac{\partial v}{\partial \xi_3} \right |_0 \label{eq-norm-V-*}
\end{equation}

\begin{remark} \label{Remark-ineq-Poicare-V-*}
Since $v = 0$ on $\xi_3=1$ for $v \in V^*$, we deduce the following Poincar\'e inequality: 
\begin{equation}
\left | v \right |_0 \le C \left | \frac{\partial v}{\partial \xi_3} \right |_0, \quad \forall v \in V^* \quad (C>0) 
\label{eq-ineq-Poincare-V-*}
\end{equation}
so \eqref{eq-norm-V-*} is indeed a norm in $V^*$, and the traces of $v \in V^*$ on $\xi_3=1$ and on $\xi_3=0$ are correctly defined, using the continuity of the trace operator from $V^*$ to $H^{-1/2}(S_0^+)$ and $H^{-1/2}(S_1^+)$ (see, for example, \cite{DautrayLionsV8} or \cite{Dumont78}).
\end{remark}

\begin{theorem}
There exist $\vec{u}^* \in \left ( V^* \right )^3$ and a subsequence of $\vec{u}^\varepsilon$ (still denoted by $\vec{u}^\varepsilon$) such that $\vec{u}^\varepsilon$ converges weakly in $\left ( V^* \right )^3$ to $\vec{u}^*$, that is, 
\begin{align}
&\vec{u}^\varepsilon \rightharpoonup \vec{u}^* \textrm{ in } \left ( L^2(\Omega) \right )^3-\textrm{weak} 
\label{eq-th-conv-1} \\
&\frac{\partial \vec{u}^\varepsilon}{\partial \xi_3} \rightharpoonup \frac{\partial \vec{u}^*}{\partial \xi_3} 
\textrm{ in } \left ( L^2(\Omega) \right )^3-\textrm{weak} \label{eq-th-conv-2} 
\end{align}
Moreover,
\begin{equation}
\varepsilon \frac{\partial \vec{u}^\varepsilon}{\partial \xi_i} \rightharpoonup \vec{0}  
\textrm{ in } \left ( L^2(\Omega) \right )^3-\textrm{weak} 
\quad (i=1,2)\label{eq-th-conv-4} 
\end{equation}
and 
\begin{align}
&\vec{u}^* = R\omega \sin \xi_1 \vec{e}_\theta \textrm{ on } \xi_3 = 0 \label{eq-th-conv-5} \\
&\vec{u}^* = \vec{0} \textrm{ on } \xi_3 = 1 \label{eq-th-conv-6} 
\end{align}
\end{theorem}
\begin{proof}
From \eqref{eq-estimate-u-r} and \eqref{eq-estimate-r} we derive
\begin{equation}
\left \| \vec{u}^\varepsilon \right \|_3 \le \frac{C}{\sqrt{\varepsilon}} \label{eq-estimate-u-3}
\end{equation}
and applying \eqref{eq-ineq-vec-2} we obtain 
\begin{equation}
\left \| \vec{u}^\varepsilon \right \|_1 \le \frac{C}{\sqrt{\varepsilon}} \label{eq-estimate-u-1}
\end{equation}
that is 
\begin{equation}
\sum_{i=1}^3 \left ( \left | \frac{\partial u_\varphi^\varepsilon}{\partial \zeta_i} \right |_1^2  
+ \left | \frac{\partial u_\theta^\varepsilon}{\partial \zeta_i} \right |_1^2  
+ \left | \frac{\partial u_\rho^\varepsilon}{\partial \zeta_i} \right |_1^2 \right ) 
\le \frac{C}{\varepsilon} 
\end{equation}
and then, after the change of variables to $\Omega_0$,
\begin{align}
&\sum_{i=1}^2  \left ( \varepsilon \left | \frac{\partial u_\varphi^\varepsilon}{\partial \xi_i} \right |_0^2  
+ \varepsilon \left | \frac{\partial u_\theta^\varepsilon}{\partial \xi_i} \right |_0^2  
+ \varepsilon \left | \frac{\partial u_\rho^\varepsilon}{\partial \xi_i} \right |_0^2 \right ) \nonumber \\
&\quad {} + \frac{1}{\varepsilon} \left ( \left | \frac{\partial u_\varphi^\varepsilon}{\partial \xi_3} \right |_0^2  
+ \left | \frac{\partial u_\theta^\varepsilon}{\partial \xi_3} \right |_0^2  
+ \left | \frac{\partial u_\rho^\varepsilon}{\partial \xi_3} \right |_0^2 \right ) 
\le \frac{C}{\varepsilon} \label{eq-estimate-Omega-0} 
\end{align}
From \eqref{eq-estimate-Omega-0}, we deduce that  
\begin{equation}
\left | \frac{\partial \vec{u}^\varepsilon}{\partial \xi_3} \right |_0 \le C, \quad 
\varepsilon \left | \frac{\partial \vec{u}^\varepsilon}{\partial \xi_1} \right |_0 \le C, \quad 
\varepsilon \left | \frac{\partial \vec{u}^\varepsilon}{\partial \xi_2} \right |_0 \le C \label{eq-estimate-derivatives}
\end{equation}
that implies the existence of a subsequence of $\vec{u}^\varepsilon$ (still denoted by $\vec{u}^\varepsilon$),  $\vec{u}^* \in \left ( V^* \right )^3$, $\vec{v}_1^* \in \left ( L^2(\Omega) \right )^3$ and $\vec{v}_2^* \in \left ( L^2(\Omega) \right )^3$, such that 
\begin{align}
&\vec{u}^\varepsilon \rightharpoonup \vec{u}^* \textrm{ in } \left ( V^* \right )^3-\textrm{weak} \label{eq-conv-1} \\
&\varepsilon \frac{\partial \vec{u}^\varepsilon}{\partial \xi_i} \rightharpoonup \vec{v}_i^*  
\textrm{ in } \left ( L^2(\Omega) \right )^3-\textrm{weak} 
\quad (i=1,2) \label{eq-conv-2}
\end{align}

From  \eqref{eq-conv-1} and \eqref{eq-g-1}--\eqref{eq-g-2} we obtain \eqref{eq-th-conv-1}--\eqref{eq-th-conv-2}. Moreover, using Remark~\ref{Remark-ineq-Poicare-V-*}, we also derive 
\eqref{eq-th-conv-5}--\eqref{eq-th-conv-6}.

To prove that $\vec{v}_1^* = \vec{0}$, let us consider $\vec{\psi} \in \left ( H_0^1(\Omega) \right )^3$. Then, from 
\eqref{eq-conv-2} it follows that
\begin{equation}
\int_\Omega \varepsilon \frac{\partial \vec{u}^\varepsilon}{\partial \xi_1} \cdot \vec{\psi}\, d\vec{x} \rightarrow 
\int_\Omega \vec{v}_1^* \cdot \vec{\psi}\, d\vec{x}, \quad \textrm{when } \varepsilon \rightarrow 0 \label{eq-conv-int-1}
\end{equation}
and from \eqref{eq-conv-1} we obtain 
\begin{equation}
\int_\Omega \varepsilon \frac{\partial \vec{u}^\varepsilon}{\partial \xi_1} \cdot \vec{\psi}\, d\vec{x} 
= - \int_\Omega \varepsilon \vec{u}^\varepsilon \cdot \frac{\partial \vec{\psi}}{\partial \xi_1}\, d\vec{x} 
\rightarrow 0,    \quad \textrm{when } \varepsilon \rightarrow 0 
\label{eq-conv-int-2}
\end{equation}

Finally, from \eqref{eq-conv-int-1}--\eqref{eq-conv-int-2}, we deduce that 
\begin{equation}
\int_\Omega \vec{v}_1^* \cdot \vec{\psi}\, d\vec{x} = 0, \quad \forall  \vec{\psi} \in \left ( H_0^1(\Omega) \right )^3 
\Rightarrow \vec{v}_1^* = \vec{0}
\end{equation}

Similarly, we show that $\vec{v}_2^* = \vec{0}$.
\end{proof}

\section{Convergence of $p^\varepsilon$} \label{section-step-6}

\begin{theorem}
There exist $p^* \in L^2(\Omega)$ and a subsequence of $p^\varepsilon$ (still denoted by 
$p^\varepsilon$) such that 
\begin{align}
&\varepsilon^2 p^\varepsilon \rightharpoonup p^* \textrm{ in } L^2(\Omega)-\textrm{weak} \label{eq-conv-p-1} \\
&\varepsilon^2 \frac{\partial p^\varepsilon}{\partial \xi_i} \rightharpoonup 
\frac{\partial p^*}{\partial \xi_i} \textrm{ in } H^{-1}(\Omega)-\textrm{weak} \quad (i=1,2) \label{eq-conv-p-2} \\
&\varepsilon^2 \frac{\partial p^\varepsilon}{\partial \xi_3} \rightharpoonup 
0 \textrm{ in } H^{-1}(\Omega)-\textrm{weak} \label{eq-conv-p-4} 
\end{align}

Furthermore, if we assume that $p^\varepsilon \in L_0^2(\Omega^\varepsilon)$ (in order to guarantee the uniqueness of the solution of \eqref{eq-Stokes-1}--\eqref{eq-Stokes-3}, or equivalently, of \eqref{eq-Stokes-var}), then
\begin{equation}
\int_D p^*\, h\, \sin \xi_1\, d\xi_1d\xi_2 = 0 
\label{eq-conv-int-p}
\end{equation}
\end{theorem}
\begin{proof}
Multiplying equation \eqref{eq-Stokes-1} by a test function $\vec{\psi} \in \left ( H_0^1(\Omega) \right )^3$ and integrating by parts, we obtain 
\begin{equation}
\mu \int_{\Omega^\varepsilon} \nabla \vec{u}^\varepsilon \cdot \nabla \vec{\psi} \, d\vec{x} 
+ \int_{\Omega^\varepsilon} \nabla p^\varepsilon \cdot \vec{\psi} \, d\vec{x} = 
0, \qquad 
\forall \vec{\psi} \in \left ( H^1_0(\Omega^\varepsilon) \right )^3 \label{eq-Stokes-variational}
\end{equation}

Rewriting problem \eqref{eq-Stokes-variational} in spherical coordinates, we get
\begin{align}
&\mu \int_{\Omega_2^\varepsilon} \Bigg [ \frac{\partial u_\rho^\varepsilon}{\partial \rho} 
\frac{\partial \psi_\rho}{\partial \rho} + \frac{1}{\rho^2} \left ( \frac{\partial u_\rho^\varepsilon}{\partial \varphi} 
- u_\varphi^\varepsilon \right ) \left ( \frac{\partial \psi_\rho}{\partial \varphi} - \psi_\varphi \right ) \nonumber \\
&\quad {} + \frac{1}{\rho^2} \left ( \frac{1}{\sin \varphi} \frac{\partial u_\rho^\varepsilon}{\partial \theta} - 
u_\theta^\varepsilon \right ) \left ( \frac{1}{\sin \varphi} \frac{\partial \psi_\rho}{\partial \theta} - 
\psi_\theta \right ) + \frac{\partial u_\varphi^\varepsilon}{\partial \rho} 
\frac{\partial \psi_\varphi}{\partial \rho} \nonumber \\
&\quad {} + \frac{1}{\rho^2} \left ( \frac{\partial u_\varphi^\varepsilon}{\partial \varphi} + 
u_\rho^\varepsilon \right ) \left ( \frac{\partial \psi_\varphi}{\partial \varphi} + 
\psi_\rho \right ) \nonumber \\
&\quad {} + \frac{1}{\rho^2} \left ( \frac{1}{\sin \varphi} \frac{\partial u_\varphi^\varepsilon}{\partial \theta} - 
(\cot \varphi) u_\theta^\varepsilon \right ) \left ( \frac{1}{\sin \varphi} \frac{\partial \psi_\varphi}{\partial \theta} - 
(\cot \varphi) \psi_\theta \right ) \nonumber \\
&\quad {} + \frac{\partial u_\theta^\varepsilon}{\partial \rho} 
\frac{\partial \psi_\theta}{\partial \rho} + \frac{1}{\rho^2} \frac{\partial u_\theta^\varepsilon}{\partial \varphi} 
\frac{\partial \psi_\theta}{\partial \varphi} \nonumber \\
&\quad {} + \frac{1}{\rho^2} \left ( \frac{1}{\sin \varphi} \frac{\partial u_\theta^\varepsilon}{\partial \theta} + 
u_\rho^\varepsilon + (\cot \varphi) u_\varphi^\varepsilon \right ) \left ( \frac{1}{\sin \varphi} \frac{\partial \psi_\theta}{\partial \theta} + \psi_\rho  \right . \nonumber \\
&\qquad {}  + (\cot \varphi) \psi_\varphi \Big ) \Bigg ] \rho^2 \sin \varphi\, d\rho d\varphi d\theta 
\nonumber \\
&\quad {} + \int_{\Omega_2^\varepsilon} \left [ \frac{\partial p^\varepsilon}{\partial \rho} \psi_\rho  + 
\frac{1}{\rho} \frac{\partial p^\varepsilon}{\partial \varphi} \psi_\varphi + 
\frac{1}{\rho \sin \varphi} \frac{\partial p^\varepsilon}{\partial \theta} \psi_\theta \right ] \rho^2 \sin \varphi\, d\rho d\varphi d\theta = 0, \nonumber \\
&\quad \forall \vec{\psi} \in \left ( H_0^1 (\Omega_2^\varepsilon ) \right )^3 \label{eq-stokes-sph-var}
\end{align}

Setting $\psi_\theta = \psi_\varphi = 0$ in \eqref{eq-stokes-sph-var}, we have 

\begin{align}
&\mu \int_{\Omega_2^\varepsilon} \Bigg [ \frac{\partial u_\rho^\varepsilon}{\partial \rho} 
\frac{\partial \psi_\rho}{\partial \rho} + \frac{1}{\rho^2} \left ( \frac{\partial u_\rho^\varepsilon}{\partial \varphi} 
- u_\varphi^\varepsilon \right ) \frac{\partial \psi_\rho}{\partial \varphi} \nonumber \\
&\quad {} + \frac{1}{\rho^2 \sin \varphi} \left ( \frac{1}{\sin \varphi} \frac{\partial u_\rho^\varepsilon}{\partial \theta} - 
u_\theta^\varepsilon \right ) \frac{\partial \psi_\rho}{\partial \theta} 
+ \frac{1}{\rho^2} \left ( \frac{\partial u_\varphi^\varepsilon}{\partial \varphi} + 
u_\rho^\varepsilon \right ) \psi_\rho \nonumber \\
&\quad {} + \frac{1}{\rho^2} \left ( \frac{1}{\sin \varphi} \frac{\partial u_\theta^\varepsilon}{\partial \theta} + 
u_\rho^\varepsilon + (\cot \varphi) u_\varphi^\varepsilon \right ) \psi_\rho \Bigg ] \rho^2 \sin \varphi\, d\rho d\varphi d\theta  \nonumber \\
&\quad {} + \int_{\Omega_2^\varepsilon} \frac{\partial p^\varepsilon}{\partial \rho} \psi_\rho  \rho^2 \sin \varphi\, d\rho d\varphi d\theta = 0, \nonumber \\
&\quad \forall \psi_\rho \in H_0^1 (\Omega_2^\varepsilon )  \label{eq-stokes-sph-var-rho}
\end{align}

If we take $\psi_\rho = \psi_\theta = 0$ in \eqref{eq-stokes-sph-var}, we obtain 
\begin{align}
&\mu \int_{\Omega_2^\varepsilon} \Bigg [  - \frac{1}{\rho^2} \left ( \frac{\partial u_\rho^\varepsilon}{\partial \varphi} 
- u_\varphi^\varepsilon \right ) \psi_\varphi + \frac{\partial u_\varphi^\varepsilon}{\partial \rho} 
\frac{\partial \psi_\varphi}{\partial \rho} + \frac{1}{\rho^2} \left ( \frac{\partial u_\varphi^\varepsilon}{\partial \varphi} + 
u_\rho^\varepsilon \right ) \frac{\partial \psi_\varphi}{\partial \varphi}  \nonumber \\
&\quad {} + \frac{1}{\rho^2\sin \varphi} \left ( \frac{1}{\sin \varphi} \frac{\partial u_\varphi^\varepsilon}{\partial \theta} - 
(\cot \varphi) u_\theta^\varepsilon \right ) \frac{\partial \psi_\varphi}{\partial \theta} \nonumber \\
&\quad {} + \frac{\cot \varphi}{\rho^2} \left ( \frac{1}{\sin \varphi} \frac{\partial u_\theta^\varepsilon}{\partial \theta} + 
u_\rho^\varepsilon + (\cot \varphi) u_\varphi^\varepsilon \right ) \psi_\varphi \Bigg ] \rho^2 \sin \varphi\, d\rho d\varphi d\theta 
\nonumber \\
&\quad {} + \int_{\Omega_2^\varepsilon} \frac{1}{\rho} \frac{\partial p^\varepsilon}{\partial \varphi} \psi_\varphi   \rho^2 \sin \varphi\, d\rho d\varphi d\theta = 0, \nonumber \\
&\quad \forall \psi_\varphi \in H_0^1 (\Omega_2^\varepsilon ) \label{eq-stokes-sph-var-varphi}
\end{align}

Choosing $\psi_\rho = \psi_\varphi = 0$ in \eqref{eq-stokes-sph-var}, it follows that 
\begin{align}
&\mu \int_{\Omega_2^\varepsilon} \Bigg [ - \frac{1}{\rho^2} \left ( \frac{1}{\sin \varphi} \frac{\partial u_\rho^\varepsilon}{\partial \theta} - 
u_\theta^\varepsilon \right ) \psi_\theta - \frac{\cot \varphi}{\rho^2} \left ( \frac{1}{\sin \varphi} \frac{\partial u_\varphi^\varepsilon}{\partial \theta} - 
(\cot \varphi) u_\theta^\varepsilon \right ) \psi_\theta  \nonumber \\
&\quad {} + \frac{\partial u_\theta^\varepsilon}{\partial \rho} 
\frac{\partial \psi_\theta}{\partial \rho} + \frac{1}{\rho^2} \frac{\partial u_\theta^\varepsilon}{\partial \varphi} 
\frac{\partial \psi_\theta}{\partial \varphi} \nonumber \\
&\quad {} + \frac{1}{\rho^2\sin \varphi} \left ( \frac{1}{\sin \varphi} \frac{\partial u_\theta^\varepsilon}{\partial \theta} + 
u_\rho^\varepsilon + (\cot \varphi) u_\varphi^\varepsilon \right ) \frac{\partial \psi_\theta}{\partial \theta}  \Bigg ] \rho^2 \sin \varphi\, d\rho d\varphi d\theta 
\nonumber \\
&\quad {} + \int_{\Omega_2^\varepsilon} \frac{1}{\rho \sin \varphi} \frac{\partial p^\varepsilon}{\partial \theta} \psi_\theta  \rho^2 \sin \varphi\, d\rho d\varphi d\theta = 0, \nonumber \\
&\quad \forall \psi_\theta \in H_0^1 (\Omega_2^\varepsilon ) \label{eq-stokes-sph-var-theta}
\end{align}

Rewriting the equations \eqref{eq-stokes-sph-var-rho}--\eqref{eq-stokes-sph-var-theta} in $\Omega$, we arrive at 
\begin{align}
&\int_{\Omega} \frac{1}{\varepsilon h} \frac{\partial p^\varepsilon}{\partial \xi_3} \psi_\rho  (R + \varepsilon \xi_3 h)^2 \sin \xi_1 \, \varepsilon h\, d\vec{\xi} = - \mu \int_{\Omega} \Bigg [ \frac{1}{\varepsilon^2 h^2} \frac{\partial u_\rho^\varepsilon}{\partial \xi_3} 
\frac{\partial \psi_\rho}{\partial \xi_3} \nonumber \\
&\quad {} + \frac{1}{(R + \varepsilon \xi_3 h)^2} \left ( \frac{\partial u_\rho^\varepsilon}{\partial \xi_1} - \frac{\xi_3}{h} 
\frac{\partial h}{\partial \xi_1} \frac{\partial u_\rho^\varepsilon}{\partial \xi_3}
- u_\varphi^\varepsilon \right ) \left ( \frac{\partial \psi_\rho}{\partial \xi_1} - \frac{\xi_3}{h} 
\frac{\partial h}{\partial \xi_1} \frac{\partial \psi_\rho}{\partial \xi_3} \right ) \nonumber \\
&\quad {} + \frac{1}{(R + \varepsilon \xi_3 h)^2 \sin \xi_1} \left ( \frac{1}{\sin \xi_1} \left ( \frac{\partial u_\rho^\varepsilon}{\partial \xi_2} - \frac{\xi_3}{h} 
\frac{\partial h}{\partial \xi_2} \frac{\partial u_\rho^\varepsilon}{\partial \xi_3} \right ) - 
u_\theta^\varepsilon \right ) \left ( \frac{\partial \psi_\rho}{\partial \xi_2} - \frac{\xi_3}{h} 
\frac{\partial h}{\partial \xi_2}\frac{\partial \psi_\rho}{\partial \xi_3} \right ) \nonumber \\
&\quad {} + \frac{1}{(R + \varepsilon \xi_3 h)^2} \left ( \frac{\partial u_\varphi^\varepsilon}{\partial \xi_1} 
- \frac{\xi_3}{h} 
\frac{\partial h}{\partial \xi_1} \frac{\partial u_\varphi^\varepsilon}{\partial \xi_3}+ 
u_\rho^\varepsilon \right ) \psi_\rho \nonumber \\
&\quad {} + \frac{1}{(R + \varepsilon \xi_3 h)^2} \left ( \frac{1}{\sin \varphi}\left (  \frac{\partial u_\theta^\varepsilon}{\partial \xi_2} - \frac{\xi_3}{h} 
\frac{\partial h}{\partial \xi_2} \frac{\partial u_\theta^\varepsilon}{\partial \xi_3} \right ) + 
u_\rho^\varepsilon \right . \nonumber \\
&\qquad {} + (\cot \varphi) u_\varphi^\varepsilon \Bigg ) \psi_\rho \Bigg ] (R + \varepsilon \xi_3 h)^2 \sin \xi_1\, \varepsilon h\, d\vec{\xi}, \nonumber \\
&\quad \forall \psi_\rho \in H_0^1 (\Omega)  \label{eq-stokes-sph-var-rho-Omega}
\end{align}

\begin{align}
&\int_{\Omega} \frac{1}{(R + \varepsilon \xi_3 h)} \left ( \frac{\partial p^\varepsilon}{\partial \xi_1} 
- \frac{\xi_3}{h} \frac{\partial h}{\partial \xi_1} \frac{\partial p^\varepsilon}{\partial \xi_3}
\right ) \psi_\varphi   (R + \varepsilon \xi_3 h)^2 \sin \xi_1\, \varepsilon h\, d\vec{\xi} \nonumber \\
&\quad = - \mu \int_{\Omega} \Bigg [ - \frac{1}{(R + \varepsilon \xi_3 h)^2} \left ( 
\frac{\partial u_\rho^\varepsilon}{\partial \xi_1} - \frac{\xi_3}{h} \frac{\partial h}{\partial \xi_1} 
\frac{\partial u_\rho^\varepsilon}{\partial \xi_3}  - u_\varphi^\varepsilon \right ) \psi_\varphi 
+ \frac{1}{\varepsilon^2 h^2} \frac{\partial u_\varphi^\varepsilon}{\partial \xi_3} 
\frac{\partial \psi_\varphi}{\partial \xi_3} \nonumber \\
&\quad {} + \frac{1}{(R + \varepsilon \xi_3 h)^2} \left ( \frac{\partial u_\varphi^\varepsilon}{\partial \xi_1} 
- \frac{\xi_3}{h} \frac{\partial h}{\partial \xi_1} \frac{\partial u_\varphi^\varepsilon}{\partial \xi_3}
+ u_\rho^\varepsilon \right ) \left ( \frac{\partial \psi_\varphi}{\partial \xi_1}
- \frac{\xi_3}{h} \frac{\partial h}{\partial \xi_1} \frac{\partial \psi_\varphi}{\partial \xi_3} \right ) \nonumber \\
&\quad {} + \frac{1}{(R + \varepsilon \xi_3 h)^2\sin \varphi} \left ( \frac{1}{\sin \varphi} \left ( \frac{\partial u_\varphi^\varepsilon}{\partial \xi_2} - \frac{\xi_3}{h} \frac{\partial h}{\partial \xi_2} 
\frac{\partial u_\varphi^\varepsilon}{\partial \xi_3} \right ) 
\right . \nonumber \\
&\qquad {} - 
(\cot \varphi) u_\theta^\varepsilon \Bigg ) \left ( \frac{\partial \psi_\varphi}{\partial \xi_2} 
- \frac{\xi_3}{h} \frac{\partial h}{\partial \xi_2} \frac{\partial \psi_\varphi}{\partial \xi_3}
\right ) \nonumber \\
&\quad {} + \frac{\cot \varphi}{(R + \varepsilon \xi_3 h)^2} \left ( \frac{1}{\sin \varphi} \left ( 
\frac{\partial u_\theta^\varepsilon}{\partial \xi_2} - \frac{\xi_3}{h} \frac{\partial h}{\partial \xi_2} 
\frac{\partial u_\theta^\varepsilon}{\partial \xi_3}
\right ) + 
u_\rho^\varepsilon \right . \nonumber \\ 
&\qquad {} + (\cot \varphi) u_\varphi^\varepsilon \Bigg ) \psi_\varphi \Bigg ] 
(R + \varepsilon \xi_3 h)^2 \sin \xi_1 \, \varepsilon h\, d\vec{\xi}
\nonumber \\
&\quad \forall \psi_\varphi \in H_0^1 (\Omega) \label{eq-stokes-sph-var-varphi-Omega}
\end{align}

\begin{align}
&\int_{\Omega} \frac{1}{(R + \varepsilon \xi_3 h) \sin \xi_1} 
\left ( \frac{\partial p^\varepsilon}{\partial \xi_2} - \frac{\xi_3}{h} \frac{\partial h}{\partial \xi_2} 
\frac{\partial p^\varepsilon}{\partial \xi_3} \right ) 
\psi_\theta  (R + \varepsilon \xi_3 h)^2 \sin \xi_1\, \varepsilon h\, d\vec{\xi} \nonumber \\
&\quad = - \mu \int_{\Omega} \Bigg [ 
- \frac{1}{(R + \varepsilon \xi_3 h)^2} \left ( \frac{1}{\sin \xi_1} 
\left ( \frac{\partial u_\rho^\varepsilon}{\partial \xi_2} - \frac{\xi_3}{h} \frac{\partial h}{\partial \xi_2} 
\frac{\partial u_\rho^\varepsilon}{\partial \xi_3} \right ) - u_\theta^\varepsilon \right ) \psi_\theta 
\nonumber \\
&\quad {} - \frac{\cot \varphi}{(R + \varepsilon \xi_3 h)^2} \left ( \frac{1}{\sin \xi_1} \left ( 
\frac{\partial u_\varphi^\varepsilon}{\partial \xi_2} - \frac{\xi_3}{h} \frac{\partial h}{\partial \xi_2} 
\frac{\partial u_\varphi^\varepsilon}{\partial \xi_3} \right ) - 
(\cot \varphi) u_\theta^\varepsilon \right ) \psi_\theta  \nonumber \\
&\quad {} + \frac{1}{\varepsilon^2 h^2} \frac{\partial u_\theta^\varepsilon}{\partial \xi_3} 
\frac{\partial \psi_\theta}{\partial \xi_3} + \frac{1}{(R + \varepsilon \xi_3 h)^2} 
\left ( \frac{\partial u_\theta^\varepsilon}{\partial \xi_1} - \frac{\xi_3}{h} \frac{\partial h}{\partial \xi_1} 
\frac{\partial u_\theta^\varepsilon}{\partial \xi_3} \right )
\left ( \frac{\partial \psi_\theta}{\partial \xi_1} - \frac{\xi_3}{h} \frac{\partial h}{\partial \xi_1} 
\frac{\partial \psi_\theta}{\partial \xi_3} \right )
\nonumber \\
&\quad {} + \frac{1}{(R + \varepsilon \xi_3 h)^2 \sin \xi_1} \left ( \frac{1}{\sin \xi_1} \left ( 
\frac{\partial u_\theta^\varepsilon}{\partial \xi_2} - \frac{\xi_3}{h} \frac{\partial h}{\partial \xi_2} 
\frac{\partial u_\theta^\varepsilon}{\partial \xi_3} \right ) + 
u_\rho^\varepsilon \right . \nonumber \\
&\qquad {}  + (\cot \varphi) u_\varphi^\varepsilon \bigg ) 
\left ( 
\frac{\partial \psi_\theta}{\partial \xi_2} - \frac{\xi_3}{h} \frac{\partial h}{\partial \xi_2} 
\frac{\partial \psi_\theta}{\partial \xi_3} \right ) \Bigg ] 
(R + \varepsilon \xi_3 h)^2 \sin \xi_1 \, \varepsilon h\, d\vec{\xi} \nonumber \\
&\quad \forall \psi_\theta \in H_0^1 (\Omega) \label{eq-stokes-sph-var-theta-Omega}
\end{align}

From \eqref{eq-stokes-sph-var-rho-Omega}, together with \eqref{eq-estimate-derivatives}--\eqref{eq-conv-1}, we deduce the existence of constants $C_1>0$ and $C_2 > 0$, 
independent of $\varepsilon$,  
such that 
\begin{align}
&\left | \int_{\Omega} \frac{\partial p^\varepsilon}{\partial \xi_3} \psi_\rho  (R + \varepsilon \xi_3 h)^2 \sin \xi_1 \, 
d\vec{\xi} \, \right | \le \frac{C_1}{\varepsilon} \left | \frac{\partial u_\rho^\varepsilon}{\partial \xi_3} \right |_0 
\left | \frac{\partial \psi_\rho}{\partial \xi_3} \right |_0 + C_2 \varepsilon \left \| \vec{u}^\varepsilon \right \|_0 
\left \| \vec{\psi} \right \|_0 \nonumber \\
&\forall  \vec{\psi} \in \left ( H_0^1(\Omega) \right )^3 \label{eq-estimate-p-1} 
\end{align}

If in \eqref{eq-estimate-p-1} we take $\displaystyle \vec{\psi} = \frac{1}{(R + \varepsilon \xi_3 h)^2 \sin \xi_1}\vec{\tilde{\psi}}$, with $\vec{\tilde{\psi}} \in \left ( H_0^1 (\Omega) \right )^3$, it follows that 
\begin{equation}
\left | \int_{\Omega} \frac{\partial p^\varepsilon}{\partial \xi_3} \tilde{\psi}_\rho  \, 
d\vec{\xi} \, \right | \le \left ( \frac{C_1}{\varepsilon} \left | \frac{\partial u_\rho^\varepsilon}{\partial \xi_3} \right |_0 
+ C_2 \varepsilon \left \| \vec{u}^\varepsilon \right \|_0 \right )
\left \| \vec{\tilde{\psi}} \right \|_0 \qquad 
\forall  \vec{\tilde{\psi}} \in \left ( H_0^1(\Omega) \right )^3 \label{eq-estimate-p-2} 
\end{equation}
which implies 
\begin{equation}
\left \| \frac{\partial p^\varepsilon}{\partial \xi_3} \right \|_{H^{-1}(\Omega)} \le 
\frac{C_1}{\varepsilon} \left | \frac{\partial u_\rho^\varepsilon}{\partial \xi_3} \right |_0 
+ C_2 \varepsilon \left \| \vec{u}^\varepsilon \right \|_0  \label{eq-estimate-p-3} 
\end{equation}
and, from \eqref{eq-estimate-derivatives}, we obtain 
\begin{equation}
\left \| \frac{\partial p^\varepsilon}{\partial \xi_3} \right \|_{H^{-1}(\Omega)} \le 
\frac{C}{\varepsilon} \label{eq-estimate-p-4}
\end{equation}

In a similar way, we infer from  \eqref{eq-stokes-sph-var-varphi-Omega} that 
there exist constants $C_1>0$, $C_2>0$ and $C_3>0$ independent of $\varepsilon$, such that 
\begin{align}
&\left | \int_{\Omega} \frac{\partial p^\varepsilon}{\partial \xi_1} 
\psi_\varphi  \, h\, (R + \varepsilon \xi_3 h) \sin \xi_1\, d\vec{\xi} \, \right | \nonumber \\
&\quad \le \left | \int_{\Omega}  \xi_3 \frac{\partial h}{\partial \xi_1} \frac{\partial p^\varepsilon}{\partial \xi_3}
\psi_\varphi   \, (R + \varepsilon \xi_3 h) \sin \xi_1\,  d\vec{\xi} \, \right | \nonumber \\
&\quad {} + \frac{C_1}{\varepsilon^2} \left | \frac{\partial u_\varphi^\varepsilon}{\partial \xi_3} \right |_0 
\left | \frac{\partial \psi_\varphi}{\partial \xi_3} \right |_0 + C_2 \left \| \vec{u}^\varepsilon \right \|_0 
\left \| \vec{\psi} \right \|_0 \nonumber \\
&\quad \le C_3 \left \| \frac{\partial p^\varepsilon}{\partial \xi_3} \right \|_{H^{-1}(\Omega)} 
\left \| \vec{\psi} \right \|_0
+ \frac{C_1}{\varepsilon^2} \left | \frac{\partial u_\varphi^\varepsilon}{\partial \xi_3} \right |_0 
\left | \frac{\partial \psi_\varphi}{\partial \xi_3} \right |_0 + C_2 \left \| \vec{u}^\varepsilon \right \|_0 
\left \| \vec{\psi} \right \|_0 \nonumber \\
&\forall  \vec{\psi} \in \left ( H_0^1(\Omega) \right )^3 \label{eq-estimate-p-5} 
\end{align}

Taking $\displaystyle \vec{\psi} = \frac{1}{h\, (R + \varepsilon \xi_3 h) \sin \xi_1}\vec{\tilde{\psi}}$, with $\vec{\tilde{\psi}} \in \left ( H_0^1 (\Omega) \right )^3$, in \eqref{eq-estimate-p-5}, we derive, as done previously, 
\begin{align}
&\left | \int_{\Omega} \frac{\partial p^\varepsilon}{\partial \xi_1} 
\tilde{\psi}_\varphi  \, d\vec{\xi} \, \right | 
\le \left ( C_3 \left \| \frac{\partial p^\varepsilon}{\partial \xi_3} \right \|_{H^{-1}(\Omega)} 
+ \frac{C_1}{\varepsilon^2} \left | \frac{\partial u_\varphi^\varepsilon}{\partial \xi_3} \right |_0 
 + C_2 \left \| \vec{u}^\varepsilon \right \|_0 \right ) 
\left \| \vec{\tilde{\psi}} \right \|_0 \nonumber \\
&\forall  \vec{\tilde{\psi}} \in \left ( H_0^1(\Omega) \right )^3 \label{eq-estimate-p-6} 
\end{align}
so we deduce that 
\begin{equation}
\left \| \frac{\partial p^\varepsilon}{\partial \xi_1} \right \|_{H^{-1}(\Omega)} 
\le C_3 \left \| \frac{\partial p^\varepsilon}{\partial \xi_3} \right \|_{H^{-1}(\Omega)} 
+ \frac{C_1}{\varepsilon^2} \left | \frac{\partial u_\varphi^\varepsilon}{\partial \xi_3} \right |_0 
 + C_2 \left \| \vec{u}^\varepsilon \right \|_0 \label{eq-estimate-p-7} 
\end{equation}
and from \eqref{eq-estimate-derivatives} and \eqref{eq-estimate-p-4}
\begin{equation}
\left \| \frac{\partial p^\varepsilon}{\partial \xi_1} \right \|_{H^{-1}(\Omega)}  \le 
\frac{C}{\varepsilon^2} \label{eq-estimate-p-8} 
\end{equation}

Repeating the same steps, starting from \eqref{eq-stokes-sph-var-theta-Omega}, we obtain that there exist constants independent of $\varepsilon$ such that
\begin{align}
&\left \| \frac{\partial p^\varepsilon}{\partial \xi_2} \right \|_{H^{-1}(\Omega)} 
\le C_1 \left \| \frac{\partial p^\varepsilon}{\partial \xi_3} \right \|_{H^{-1}(\Omega)} 
+ \frac{C_2}{\varepsilon^2} \left | \frac{\partial u_\theta^\varepsilon}{\partial \xi_3} \right |_0 
 + C_3 \left \| \vec{u}^\varepsilon \right \|_0 \label{eq-estimate-p-9} \\
&\left \| \frac{\partial p^\varepsilon}{\partial \xi_2} \right \|_{H^{-1}(\Omega)}  \le 
\frac{C}{\varepsilon^2} \label{eq-estimate-p-10} 
\end{align}

Finally, the existence of a subsequence of $p^\varepsilon$ satisfying \eqref{eq-conv-p-2}--\eqref{eq-conv-p-4} is a consequence of estimates \eqref{eq-estimate-p-4}, \eqref{eq-estimate-p-8} and \eqref{eq-estimate-p-10}. Moreover, \eqref{eq-conv-p-1} is deduced from \eqref{eq-conv-p-2}--\eqref{eq-conv-p-4} by applying the Ne\v{c}as inequality (see, for example, \cite{Ciarlet2015}).

To complete the proof, assuming that $p^\varepsilon \in L_0^2(\Omega^\varepsilon)$, we have that 
\begin{equation}
\int_{\Omega^\varepsilon} p^\varepsilon\, d\vec{x} = 0
\end{equation}

Then, making a change of variables to $\Omega$, we get
\begin{equation}
\int_\Omega p^\varepsilon (R+\varepsilon \xi_3 h)^2 \sin \xi_1 \, \varepsilon h\, d\vec{\xi} = 0 
\end{equation}

Multiplying by $\varepsilon$ and passing to the limit, we arrive at 
\begin{equation}
\int_\Omega p^* h \sin \xi_1\, d\vec{\xi} = 0 \label{eq-limit-int-p}
\end{equation}
and, since the integrand in \eqref{eq-limit-int-p} does not depend on $\xi_3$, we conclude \eqref{eq-conv-int-p}.

\end{proof}

\section{Determining  the limits $\vec{u}^*$ and $p^*$} \label{section-step-7}

\begin{theorem}
The components of the limit velocity field $\vec{u}^*$ fulfill: 
\begin{align}
u^*_\varphi &= \frac{h^2}{2 \mu R} \frac{\partial p^*}{\partial \xi_1} \xi_3 ( \xi_3 - 1) \label{eq-det-u-5} \\
u^*_\theta &= \frac{h^2}{2 \mu R \sin \xi_1} \frac{\partial p^*}{\partial \xi_2} \xi_3 ( \xi_3 - 1) 
+ R \omega \sin \xi_1 ( 1 - \xi_3) \label{eq-det-u-6} \\
u^*_\rho &= 0 \label{eq-det-u-6-b}
\end{align}
\end{theorem}
\begin{proof}
Let us multiply equation \eqref{eq-stokes-sph-var-varphi-Omega} by $\varepsilon$ and let $\varepsilon \rightarrow 0$. Taking into account \eqref{eq-th-conv-1}--\eqref{eq-th-conv-4} and \eqref{eq-conv-p-1}--\eqref{eq-conv-p-4} we derive 
\begin{align}
&\int_\Omega R \sin \xi_1 h \frac{\partial p^*}{\partial \xi_1} \psi_\varphi \, d\vec{\xi} 
= - \mu \int_\Omega \frac{R^2 \sin \xi_1}{h} \frac{\partial u^*_\varphi}{\partial \xi_3} 
\frac{\partial \psi_\varphi}{\partial \xi_3} \, d\vec{\xi}, \nonumber \\
&\forall \psi_\varphi \in H_0^1(\Omega) \label{eq-det-u-1}
\end{align}

Since $\psi_\varphi \in H_0^1(\Omega)$, it follows from \eqref{eq-det-u-1} that 
\begin{align}
&\int_\Omega \left ( R \sin \xi_1 h \frac{\partial p^*}{\partial \xi_1} \right ) \psi_\varphi \, d\vec{\xi} 
= \mu \int_\Omega \frac{\partial}{\partial \xi_3} \left ( \frac{R^2 \sin \xi_1}{h} \frac{\partial u^*_\varphi}{\partial \xi_3} 
\right ) \psi_\varphi \, d\vec{\xi}, \nonumber \\
&\forall \psi_\varphi \in H_0^1(\Omega) \label{eq-det-u-2}
\end{align}
thus, 
\begin{equation}
R \sin \xi_1 h \frac{\partial p^*}{\partial \xi_1} = \mu\, \frac{\partial}{\partial \xi_3} \left ( \frac{R^2 \sin \xi_1}{h} \frac{\partial u^*_\varphi}{\partial \xi_3} \right ) \label{eq-det-u-3}
\end{equation}
and then 
\begin{equation}
\frac{\partial^2 u^*_\varphi}{\partial \xi_3^2} = \frac{h^2}{\mu R} \frac{\partial p^*}{\partial \xi_1} 
\quad \textrm{in} \  H^{-1}(\Omega) \label{eq-det-u-4}
\end{equation}

From \eqref{eq-det-u-4} and the boundary conditions \eqref{eq-th-conv-5}--\eqref{eq-th-conv-6}, we deduce 
\eqref{eq-det-u-5}. Repeating the same steps starting from the equation \eqref{eq-stokes-sph-var-theta-Omega}, we arrive at 
\eqref{eq-det-u-6}.

Writing equation \eqref{eq-Stokes-2} in $\Omega$, we obtain 
an equation like \eqref{eq-continuity-omega}, with $\vec{u}^\varepsilon$ instead of $\vec{v}^\varepsilon$. Now multiplying this equation by $\psi \in L^2(\Omega)$ and integrating over $\Omega$, we get
\begin{align}
&\int_\Omega \Biggl [ \frac{1}{\varepsilon h (R+\varepsilon \xi_3 h)^2}\frac{\partial}{\partial \xi_3} \left [ (R+\varepsilon \xi_3 h)^2 u_{\rho}^\varepsilon \right ] \nonumber \\
&\quad {} + \frac{1}{(R+\varepsilon \xi_3 h) \sin \xi_1}\left [ \frac{\partial}{\partial \xi_1} \left ( u_{\varphi}^\varepsilon \sin \xi_1 \right ) - \frac{\xi_3}{h}\frac{\partial h}{\partial \xi_1} \frac{\partial}{\partial \xi_3} \left ( u_{\varphi}^\varepsilon \sin \xi_1 \right ) \right ] \nonumber \\
&\quad {} + 
\frac{1}{(R+\varepsilon \xi_3 h) \sin \xi_1}\left [ \frac{\partial u_{\theta}^\varepsilon}{\partial \xi_2} 
- \frac{\xi_3}{h}\frac{\partial h}{\partial \xi_2} \frac{\partial u_{\theta}^\varepsilon}{\partial \xi_3} \right ] 
\Biggr ] \psi \, d\vec{\xi} 
= 0, \nonumber \\
&\qquad \forall \psi \in L^2(\Omega) \label{eq-continuity-omega-2}
\end{align}

After multiplying \eqref{eq-continuity-omega-2} by $\varepsilon$, letting $\varepsilon \rightarrow 0$, 
and using \eqref{eq-th-conv-1}--\eqref{eq-th-conv-4}, it follows 
\begin{equation}
\int_\Omega \frac{1}{h} \frac{\partial u^*_\rho}{\partial \xi_3} \psi \, d\vec{\xi} = 0, \qquad 
\forall \psi \in L^2(\Omega) \label{eq-limit-u-rho} 
\end{equation}

From \eqref{eq-limit-u-rho} and boundary condition \eqref{eq-th-conv-6}, \eqref{eq-det-u-6-b} is yielded.

\end{proof}

\begin{remark}
From \eqref{eq-estimate-p-4} and \eqref{eq-stokes-sph-var-rho-Omega}, we can deduce, in a similar way as above, the existence of a subsequence of $p^\varepsilon$ such that 
\begin{equation}
\varepsilon \frac{\partial p^\varepsilon}{\partial \xi_3} \rightharpoonup \frac{\mu}{h} \frac{\partial^2 u^*_\rho}{\partial \xi_3^2} 
\textrm{ in } H^{-1}(\Omega)-\textrm{weak} \label{eq-det-u-7} 
\end{equation}
and, from \eqref{eq-det-u-6-b}, we derive that 
\begin{equation}
\varepsilon \frac{\partial p^\varepsilon}{\partial \xi_3} \rightharpoonup 0 
\textrm{ in } H^{-1}(\Omega)-\textrm{weak} \label{eq-det-u-7-b} 
\end{equation}
\end{remark}

\begin{theorem}
Let us consider the spaces 
\begin{align}
H^1_{\#}(D_\varphi) &= \left \{ \psi \in H^1(D_\varphi) / \psi \textrm{ is } \xi_2\textrm{-periodic} \right \} \label{eq-det-p-14-b-bis} \\ 
H^1_{\#}(D) &= \left \{ \psi \in H^1(D) / \psi \textrm { is } \xi_2\textrm{-periodic} \right \} \label{eq-det-p-14-b}
\end{align}
and let us suppose that 
\begin{equation}
k_i^\varepsilon \rightarrow k_i^0 \quad \textrm{ in } H^1_{\#}(D_\varphi) \quad (i=0,1) \label{eq-conv-k-0-1}
\end{equation}
as $\varepsilon \rightarrow 0$. Then $p^* \in H^1_{\#}(D)$ is the unique solution of the Neumann problem 
\begin{align}
&p^* \in H^1_{\#}(D), \quad \int_D p^*\, h\, \sin \xi_1\, d\xi_1d\xi_2 = 0, \nonumber \\
&\int_{D} 
\left [ \frac{h^3 \sin \xi_1}{12\mu} \frac{\partial p^*}{\partial \xi_1}  \frac{\partial \psi}{\partial \xi_1} + 
\frac{h^3}{12\mu \sin \xi_1} \frac{\partial p^*}{\partial \xi_2}  \frac{\partial \psi}{\partial \xi_2} 
 \right ]  d\xi_1 d\xi_2 \nonumber \\
 &\quad = \int_D \frac{R^2 h \omega \sin \xi_1}{2} \frac{\partial \psi}{\partial \xi_2}  d\xi_1 d\xi_2 \nonumber \\
&{} + \int_0^{2\pi} R \left [ 
h(\varphi_0,\xi_2) 
\left ( \int_0^1 k_0^0(\xi_2, \xi_3)\, d\xi_3 \right ) \psi(\varphi_0, \xi_2) \sin \varphi_0 \right . \nonumber \\
&\quad {} - \left . 
h(\varphi_1,\xi_2) \left ( \int_0^1 k_1^0(\xi_2, \xi_3)\, d\xi_3 \right ) \psi(\varphi_1, \xi_2) \sin \varphi_1 
\right ] \, d\xi_2   \nonumber \\ 
&\quad \forall \psi \in H^1_{\#}(D)
\label{eq-det-p-16}
\end{align}

\end{theorem}
\begin{proof}

From \eqref{eq-Stokes-2} we deduce that 
\begin{equation}
\int_{\Omega^\varepsilon} \left ( \diver \vec{u}^\varepsilon \right ) \psi\, d\vec{x} = 0 \quad 
\forall \psi \in H^1(\Omega^\varepsilon) \label{eq-det-p-1}
\end{equation}
hence 
\begin{equation}
\int_{\Omega^\varepsilon} \vec{u}^\varepsilon \cdot \nabla \psi \, d\vec{x} =  
\int_{\partial \Omega^\varepsilon} \left ( \vec{u}^\varepsilon \cdot \vec{n} \right ) \psi \, dA
 \quad 
\forall \psi \in H^1(\Omega^\varepsilon) \label{eq-det-p-2}
\end{equation}

After a change of variable to the $(\varphi, \theta, \rho)$ system of coordinates, we identify a function $\psi \in H^1_{\#}(D)$ with a function $\psi \in H^1(\Omega_2^\varepsilon)$, independent of $\rho$, and periodic in $\theta$. Then, the left hand side of equation \eqref{eq-det-p-2} can be written as 
\begin{align}
&\int_{\Omega^\varepsilon} \vec{u}^\varepsilon \cdot \nabla \psi \, d\vec{x} \nonumber \\ 
&=\int_{\Omega_2^\varepsilon} 
\left [ \frac{u_\varphi^\varepsilon}{\rho} \frac{\partial \psi}{\partial \varphi} + 
\frac{u_\theta^\varepsilon}{\rho \sin \varphi} \frac{\partial \psi}{\partial \theta} \right ] \rho^2 \sin \varphi \, d\rho d\varphi d\theta \nonumber \\
&= \int_{\Omega_2^\varepsilon} 
\left [ \left ( u_\varphi^\varepsilon \rho \sin \varphi \right )  \frac{\partial \psi}{\partial \varphi} + 
\left ( u_\theta^\varepsilon \rho \right ) \frac{\partial \psi}{\partial \theta} \right ] d\rho d\varphi d\theta \nonumber \\
&\forall \psi \in H^1_{\#}(D) 
\label{eq-det-p-5}
\end{align}
and changing to the variable $\vec{\xi}$ 
\begin{align}
&\int_{\Omega^\varepsilon} \vec{u}^\varepsilon \cdot \nabla \psi \, d\vec{x} \nonumber \\ 
&=  \int_{\Omega} 
\left [ \left ( u_\varphi^\varepsilon (R + \varepsilon \xi_3 h ) \sin \xi_1 \right )  \frac{\partial \psi}{\partial \xi_1} + 
\left ( u_\theta^\varepsilon (R + \varepsilon \xi_3 h ) \right ) \frac{\partial \psi}{\partial \xi_2} 
 \right ] \varepsilon h \, d\vec{\xi} \nonumber \\
&\forall \psi \in H^1_{\#}(D)
\label{eq-det-p-7}
\end{align}

Due to the boundary conditions \eqref{eq-Stokes-3} and \eqref{eq-g-1}--\eqref{eq-g-5-d}, taking  $\psi \in H^1_{\#}(D)$, 
and performing a change of variables to $\vec{\xi}$, the right hand side of equation \eqref{eq-det-p-2} can be rewritten as 
\begin{align}
&\int_{\partial \Omega^\varepsilon} \left ( \vec{u}^\varepsilon \cdot \vec{n} \right ) \psi \, dA \nonumber \\
&= \varepsilon \iint_{D_\varphi} \left [ h(\varphi_1,\xi_2) \left ( R + \varepsilon \xi_3 h(\varphi_1,\xi_2) \right ) 
k_1^\varepsilon(\xi_2, \xi_3) \psi(\varphi_1, \xi_2, \xi_3) \sin \varphi_1 \right . \nonumber \\
&\quad{} \left . {} - 
h(\varphi_0,\xi_2) \left ( R + \varepsilon \xi_3 h(\varphi_0,\xi_2) \right ) 
k_0^\varepsilon(\xi_2, \xi_3) \psi(\varphi_0, \xi_2, \xi_3) \sin \varphi_0 \right ] \, d\xi_2 d\xi_3  \label{eq-det-p-7-b} 
\end{align}

Now, from \eqref{eq-det-p-2}, \eqref{eq-det-p-7} and \eqref{eq-det-p-7-b}, it follows 
\begin{align}
&\int_{\Omega} 
\left [ \left ( u_\varphi^\varepsilon (R + \varepsilon \xi_3 h ) \sin \xi_1 \right )  \frac{\partial \psi}{\partial \xi_1} + 
\left ( u_\theta^\varepsilon (R + \varepsilon \xi_3 h ) \right ) \frac{\partial \psi}{\partial \xi_2} 
 \right ] \varepsilon h \, d\vec{\xi} \nonumber \\
&\quad = \varepsilon \iint_{D_\varphi} \left [ h(\varphi_1,\xi_2) \left ( R + \varepsilon \xi_3 h(\varphi_1,\xi_2) \right ) 
k_1^\varepsilon(\xi_2, \xi_3) \psi(\varphi_1, \xi_2) \sin \varphi_1 \right . \nonumber \\
&\quad{} \left . {} - 
h(\varphi_0,\xi_2) \left ( R + \varepsilon \xi_3 h(\varphi_0,\xi_2) \right ) 
k_0^\varepsilon(\xi_2, \xi_3) \psi(\varphi_0, \xi_2) \sin \varphi_0 \right ] \, d\xi_2 d\xi_3  \nonumber \\
&\forall \psi \in H^1_{\#}(D) 
\label{eq-det-p-7-c}
\end{align}

If we now divide \eqref{eq-det-p-7-c} by $\varepsilon$ and let $\varepsilon \rightarrow 0$, taking into account \eqref{eq-conv-k-0-1}, we obtain
\begin{align}
&\int_{\Omega} 
\left [ \left ( u_\varphi^* R h \sin \xi_1 \right )  \frac{\partial \psi}{\partial \xi_1} + 
\left ( u_\theta^* R h \right ) \frac{\partial \psi}{\partial \xi_2} 
 \right ]  d\vec{\xi} \nonumber \\
&\quad = \iint_{D_\varphi} R \left [ h(\varphi_1,\xi_2)  
k_1^0(\xi_2, \xi_3) \psi(\varphi_1, \xi_2) \sin \varphi_1 \right . \nonumber \\
&\quad{} \left . {} - 
h(\varphi_0,\xi_2) 
k_0^0(\xi_2, \xi_3) \psi(\varphi_0, \xi_2) \sin \varphi_0 \right ] \, d\xi_2 d\xi_3  \nonumber \\
&\forall \psi \in H^1_{\#}(D)
\label{eq-det-p-8}
\end{align}
and then
\begin{align}
&\int_{D} 
\left [ \left ( \left ( \int_0^1 u_\varphi^*\, d\xi_3 \right ) R h \sin \xi_1 \right )  \frac{\partial \psi}{\partial \xi_1} + 
\left (  \left ( \int_0^1 u_\theta^*\, d\xi_3 \right ) R h \right ) \frac{\partial \psi}{\partial \xi_2} 
 \right ]  d\xi_1 d\xi_2 \nonumber \\
&\quad = \int_0^{2\pi} R \left [ 
h(\varphi_1,\xi_2) \left ( \int_0^1 k_1^0(\xi_2, \xi_3)\, d\xi_3 \right ) \psi(\varphi_1, \xi_2) \sin \varphi_1 \right . 
\nonumber \\
&\quad {} - \left . h(\varphi_0,\xi_2) 
\left ( \int_0^1 k_0^0(\xi_2, \xi_3)\, d\xi_3 \right ) \psi(\varphi_0, \xi_2) \sin \varphi_0
\right ] \, d\xi_2 \nonumber \\ 
&\forall \psi \in H^1_{\#}(D)
\label{eq-det-p-14}
\end{align}

Moreover, from \eqref{eq-det-u-5}--\eqref{eq-det-u-6}  we deduce that 
\begin{align}
\int_0^1 u_\varphi^*\, d\xi_3 &= - \frac{h^2}{12\mu R} \frac{\partial p^*}{\partial \xi_1} \label{eq-det-p-11} \\
\int_0^1 u_\theta^*\, d\xi_3 &= - \frac{h^2}{12\mu R \sin \xi_1} \frac{\partial p^*}{\partial \xi_2} 
+ \frac{1}{2} R\omega \sin \xi_1 \label{eq-det-p-12}
\end{align}
thus, \eqref{eq-det-p-14} yields 
\begin{align}
&\int_{D} 
\left [ \frac{h^3 \sin \xi_1}{12\mu} \frac{\partial p^*}{\partial \xi_1}  \frac{\partial \psi}{\partial \xi_1} + 
\left ( \frac{h^3}{12\mu \sin \xi_1} \frac{\partial p^*}{\partial \xi_2}  - 
\frac{R^2 h \omega \sin \xi_1}{2} \right ) \frac{\partial \psi}{\partial \xi_2} 
 \right ]  d\xi_1 d\xi_2 \nonumber \\
&{} + \int_0^{2\pi} R \left [ 
h(\varphi_1,\xi_2) \left ( \int_0^1 k_1^0(\xi_2, \xi_3)\, d\xi_3 \right ) \psi(\varphi_1, \xi_2) \sin \varphi_1 \right . 
\nonumber \\
&\quad {} - \left . h(\varphi_0,\xi_2) 
\left ( \int_0^1 k_0^0(\xi_2, \xi_3)\, d\xi_3 \right ) \psi(\varphi_0, \xi_2) \sin \varphi_0
\right ] \, d\xi_2  = 0 \nonumber \\ 
&\forall \psi \in H^1_{\#}(D)
\label{eq-det-p-15}
\end{align}

Finally, problem \eqref{eq-det-p-15} can be rewritten as \eqref{eq-det-p-16}, using 
Remark \ref{remark-bound-sin} to ensure the coercivity of the bilinear form in \eqref{eq-det-p-16}, and the integral condition \eqref{eq-conv-int-p} to guarantee the uniqueness of the solution.

\end{proof}

\begin{remark}
It is straightforward to verify that the strong formulation of problem \eqref{eq-det-p-16} is 
\begin{align}
&p^* \in H^1_{\#}(D), \quad \int_D p^*\, h\, \sin \xi_1\, d\xi_1d\xi_2 = 0, \label{eq-det-p-17-a} \\
&\frac{h^2}{12 \mu} \frac{\partial p^*}{\partial \xi_1}  = 
-R\left ( \int_0^1 k_0^0(\xi_2, \xi_3)\, d\xi_3 \right ) 
\textrm{ on } \xi_1 = \varphi_0, 
\label{eq-det-p-17-b} \\
&\frac{h^2}{12 \mu} \frac{\partial p^*}{\partial \xi_1}  = 
-R\left ( \int_0^1 k_1^0(\xi_2, \xi_3)\, d\xi_3 \right ) 
\textrm{ on } \xi_1 = \varphi_1, 
\label{eq-det-p-17-c} \\
&\frac{\partial}{\partial \xi_1} \left ( \frac{h^3 \sin \xi_1}{12\mu} \frac{\partial p^*}{\partial \xi_1} \right )   + 
\frac{\partial}{\partial \xi_2} \left ( \frac{h^3}{12\mu \sin \xi_1} \frac{\partial p^*}{\partial \xi_2} \right ) = 
\frac{R^2 \omega \sin \xi_1}{2} \frac{\partial h}{\partial \xi_2}
\label{eq-det-p-18}
\end{align}
\end{remark}

\begin{remark}
From the uniqueness of the limits $u^*_\varphi$, $u^*_\theta$, $u^*_\rho$ and $p^*$, 
we deduce that 
the convergences in \eqref{eq-th-conv-1}--\eqref{eq-th-conv-4} and \eqref{eq-conv-p-1}--\eqref{eq-conv-p-4} hold for the whole sequence, and not only for a subsequence.
\end{remark}

\begin{remark}
In problem \eqref{eq-det-p-17-a}--\eqref{eq-det-p-18}, it is possible to impose Dirichlet boundary conditions on the pressure instead of \eqref{eq-det-p-17-b}--\eqref{eq-det-p-17-c}, although their justification requires a different approach than the one used here (see, for example, \cite{Amedodjietal2002} or \cite{BayadaChambat1989}).
\end{remark}

\section*{Acknowledgements}
This work has been partially supported by Project PID2024-158035NB-I00, funded by 
MICIU/AEI/10.13039/501100011033 and FEDER, EU.

\end{document}